\documentclass[12pt]{article}
\usepackage{blindtext}
\usepackage[margin=1.2in]{geometry}
\usepackage{xcolor}
\usepackage{color,soul}
\usepackage{hyperref}
\usepackage[utf8]{inputenc}
\usepackage{amsmath}
\usepackage{amsthm}
\usepackage{amsfonts,mathrsfs}
\usepackage{amssymb}
\usepackage{graphicx}
\usepackage{enumitem}
\usepackage{color}
\usepackage{verbatim}
\usepackage[normalem]{ulem}
\theoremstyle{plain}
\newtheorem{theorem}{Theorem}[section]
\newtheorem{lemma}[theorem]{Lemma}

\newtheorem{corollary}[theorem]{Corollary}
\newtheorem{proposition}[theorem]{Proposition}

\theoremstyle{definition}
\newtheorem{definition}[theorem]{Definition}

\newtheorem{remark}[theorem]{Remark}
\usepackage{hyperref}
\hypersetup{citecolor=purple, linkcolor=blue, colorlinks=true}

\newcommand{\F}{\mathbb F}
\newcommand{\N}{\mathrm N}

\newcommand{\cL}{\mathcal L}

\newcommand{\Aut}{\mathrm{Aut}}

\newcommand{\GL}{\mathrm{GL}}

\newcommand{\PG}{\mathrm{PG}}
\newcommand{\GaL}{\Gamma\mathrm{L}}
\newcommand{\ints}{\mathrm{intn}_\sigma(\Gamma)}
\newcommand{\PGaL}{\mathrm{P}\Gamma\mathrm{L}}
\newcommand{\Tr}{ \ensuremath{ \mathrm{Tr}}}

\newcommand{\RN}[1]{%
  \textup{\uppercase\expandafter{\romannumeral#1}}%
}
\newcommand{\rn}[1]{%
  \textup{\lowercase\expandafter{\romannumeral#1}}%
}

\title{A geometric characterization of known maximum scattered linear sets of $\PG(1,q^n)$}

\author{G.G. Grimaldi, S. Gupta, G. Longobardi, R. Trombetti}

\date{}

\begin{document}

\maketitle

\begin{abstract}
\noindent An \textit{$\mathbb{F}_q$-linear set} $L=L_U$ of $\Lambda=\mathrm{PG}(V, \mathbb{F}_{q^n}) \cong \mathrm{PG}(r-1,q^n)$ is a set of points defined by non-zero vectors of an $\mathbb{F}_q$-subspace $U$ of $V$. \noindent The integer $\dim_{\mathbb{F}_q} U$ is called the \textit{rank} of $L$. In [G.~Lunardon, O.~Polverino:  Translation ovoids of orthogonal polar spaces. \textit{Forum Math.} \textbf{16} (2004)], it was proven that any $\mathbb{F}_q$-linear set $L$ of $\Lambda$ of rank $u$ such that $\langle L \rangle=\Lambda$ is either a canonical subgeometry of $\Lambda$ or there are a $(u-r-1)$-dimensional subspace $\Gamma$ of $\mathrm{PG}(u-1,q^n) \supset \Lambda$ disjoint from $\Lambda$ and a canonical subgeometry $\Sigma \cong \PG(u-1,q)$ disjoint from $\Gamma$ such that $L$ is the projection of $\Sigma$ from $\Gamma$ onto $\Lambda$. The subspace $\Gamma$ is called the \textit{vertex} of the projection. In this article, we will show a method to reconstruct the vertex $\Gamma$ for a peculiar class of linear sets of rank $u = n(r - 1)$ in $\mathrm{PG}(r - 1, q^n)$
called {\em evasive} linear sets. Also, we will use this result to characterize some families of linear sets of the projective line $\mathrm{PG}(1,q^n)$ introduced from 2018 onward, by means of certain properties of their projection vertices, as done in [B.~Csajbók, C.~Zanella: On scattered linear sets of pseudoregulus type in $\PG(1, q^t)$, \newblock{\em Finite Fields Appl.} \textbf{41} (2016)] and in  [C. Zanella, F. Zullo: Vertex properties of maximum scattered linear sets of $\PG(1, q^n)$. {\em Discrete Math.} \textbf{343}(5) (2020)].
\end{abstract}

\noindent \\

\section{Introduction}\label{sec: intro}

Let $u \geq n, r \geq 2$ be positive integers and let $\PG(r-1,q^n)$ be the projective space over the $r$-dimensional vector space $V=\F_{q^n}^r$. A subset $L_U$ of $\PG(r-1,q^n)$ is called a {\it linear set} if its points are defined by non-zero elements of an $\mathbb{F}_q$-subspace $U$ of $V$. More precisely, $$L_U=\{\langle \textbf{u}\rangle_{\mathbb{F}_{q^n}}:\textbf{u}\in U\setminus \{\boldsymbol{0}\}\}.$$ If $\text{dim}_{\F_q}(U)=u$, then we say that $L_U$ is a linear set of {\it rank} $u$, and we will denote this by the symbol $\mathrm{rk}(L_U)=u$. 
The largest possible rank of a proper  $\F_q$-linear set $L_U$ of $\PG(r-1,q^n)$ is $(r-1)n$,  otherwise $L_U=\PG(r-1,q^n)$.
Also, if $L_U$ has rank $r$ and $\langle L_U \rangle=\PG(r - 1, q^n)$, we call it
 a \textit{canonical subgeometry} of $\PG(r - 1, q^n)$. \\
Regarding the size, we have that $|L_U| \leq q^{u-1}+q^{u-2}+\dots+q+1,$ and if this bound is attained we say that $L_U$ is a {\it scattered} linear set of $\PG(r-1, q^n)$. 

Two linear sets $L_U$ and $L_W$ are \textit{$\PGaL$-equivalent} (or simply \textit{equivalent}) if there is $\phi \in \PGaL(r,q^n)$ such that $L_U^{\phi}=L_W$. In this case, we will write $L_U \cong L_W$.
It is known that any canonical subgeometry of $\PG(r-1,q^n)$ is equivalent to the linear set
\begin{equation*}
   \{ \langle (a_0,\ldots,a_{r-1}) \rangle_{\F_{q^n}} \colon (a_0,\ldots,a_{r-1}) \in \F^r_q \setminus \{\textbf{0}\} \} \cong  \PG(r-1,q).
\end{equation*}
A subspace $S=\PG(Z,q^n)$ of $\PG(V, q^n)$ has \textit{weight $i$} with respect to $L$ if $\dim_{\F_q}(Z \cap U)=i$, and we will use the following notation $\mathrm{wt}_{L}(S)=i$. 

A linear set $L \subset \PG(V,q^n)$ is said to be $(h,k)_q$-{\it evasive} if there exists a subspace $U \leq V$ with $L=L_U$ such that $\dim \langle L \rangle \geq h-1$, and for every $h$-dimensional
$\mathbb{F}_{q^n}$-subspace $T$ of $V$, we have that $L_{U \cap T}$ has rank at most $k$, \cite{Bartoli_Csajbok_Marino_Trombetti}.
Note that if $h=k=1$, then $L_U$ is scattered. 

As it was proven in 2000 by Blokhuis and Lavrauw, if $L_U$ is scattered then its rank is at most $\lfloor\frac{rn}{2}\rfloor$ and if this bound is attained, we say that $L_U$ is {\it maximum} scattered, \cite{Blokhuis-Lavrauw}. 

Lunardon and Polverino in \cite[Theorem 2]{Lunardon_Polverino} showed that every linear set of rank $u$ of $\Lambda=\PG(V,q^n)$ spanning $\Lambda$ over $\F_{q^n}$, is either a canonical subgeometry of $\Lambda$ or it can be obtained as projection of a canonical subgeometry $\Sigma \cong \PG(u-1,q)$ of $\Sigma^* =\PG(u-1,q^n) \supset \Lambda$ from a suitable subspace $\Gamma$ of $\Sigma^*$ onto $\Lambda$ such that $\Gamma \cap \Sigma = \emptyset = \Gamma \cap \Lambda$. The subspaces $\Gamma$ and $\Lambda$ are also called {\it vertex} and {\it center} of the projection, respectively. 

We recall that if $\Sigma$ is a canonical subgeometry of $\Sigma^*$, a subspace $S=\PG(Z,\F_{q^n})$ of $\Sigma^*$ is a \textit{subspace of $\Sigma$} if the dimension of $Z$ is equal to the rank of the $\F_q$-linear set $S \cap \Sigma$.  This is equivalent to say that if $\sigma$ is a collineation of $\Sigma^*$ of order $n$ fixing $\Sigma$ pointwise (i.e., $\mathrm{Fix}(\sigma)=\Sigma$), then the subspace $S$ is fixed by $\sigma$, see \cite[Lemma 1]{normal-spread}.
Let $P$ be a point of $\Sigma^*$, the subspace $\mathcal{L}_{P,\sigma}= \langle P, P^{\sigma}, \ldots, P^{\sigma^{n-1}} \rangle$ of $\Sigma^*$ clearly is a subspace of $\Sigma$ of dimension at most $n-1$.
The point $P$ will be called an \textit{imaginary} point of $\Sigma^*$, with respect to $\Sigma$, if $\dim \mathcal{L}_{P,\sigma}=n-1$.
If the collineation $\sigma$ is clear from the context, we will indicate $\mathcal{L}_{P,\sigma}$ simply by $\mathcal{L}_{P}$.
Also, we will denote by $\Aut(\Sigma)$, the \textit{automorphism group} of $\Sigma$, i.e., the group made up of all the maps in $\PGaL(u,q^n)$ which commute with $\sigma$ and we will put $\mathrm{LAut}(\Sigma)= \Aut(\Sigma) \cap \mathrm{PGL}(u,q^n)$.\\ 
In this article a method to reconstruct, up to the action of the setwise stabilizer of $\Sigma$ in $\PGaL(u,q^n)$, the vertex $\Gamma$ of an $(r-1,n(r-1)-2)_q$-evasive linear set of rank $u=n(r-1)$ in $\PG(r-1,q^n)$ from a set made up of $r-1$ imaginary points of $\PG(u-1,q^n)$ and their conjugates with respect to a collineation of $\PG(u-1,q^n)$ fixing  pointwise a suitable subgeometry $\Sigma$, will be given; see Theorem \ref{Fdirect} and \ref{GPT}. 

The case when $r=2$ and $L$ is a maximum scattered linear sets of $\PG(1,q^n)$ is of particular interest because, as proven by Sheekey in Section $5$ of \cite{sheekey}, these linear sets yield remarkable examples of Maximum Rank Distance codes (MRD codes, for short); indeed, linear MRD codes $2$-dimensional over $\F_{q^n}$ with minimum distance $n-1$, which have various important applications in communications and cryptography, see \cite{sheekey2} and  the references therein.
In this circumstance, in the same spirit of \cite{Csajbok_Zanella} and \cite{Zanella-Zullo}, elaborating on above mentioned achievement, we will characterize families of maximum scattered linear sets of $\PG(1,q^n)$ introduced from 2018 onwards (see \cite{Csajbok_Marino_Zullo, Csajbok-Marino-Polverino-Zanella, longobardi_marino_trombetti_zhou, longobardi_zanella, neri_santonastaso_zullo, SmaZaZu}) through certain properties of their projection vertices.

\section{Preliminary results}\label{sec: prel}

\noindent  Up to a suitable projectivity, any linear set $L$ of rank $n(r-1)$ in $\PG(r-1,q^n)$ can be written in the following form \begin{equation}{\label{linearset}}
    L=L_F=\{\langle({\bf x},F({\bf x}))\rangle_{\F_{q^n}} \, \colon \, {\bf x}=(x_0,x_1,...,x_{r-2}) \in \F_{q^n}^{r-1}, {\bf x} \ne {\bf 0}\},
\end{equation}
where 
\begin{equation}\label{multivariatepolynomial}
F({\bf x})=\sum_{j=0}^{r-2}f_{j}(x_j),
\end{equation}
and $f_{j}(x_j)=\sum_{i=1}^{n-1} a_{ij} x_j^{q^{si}}$
is an $\F_{q^s}$-linearized polynomial with coefficients over $\F_{q^n}$ in the indeterminate $x_j$ for each $j\in \{0,1,...,r-2\}$ and $\gcd(s,n)=1$, see \cite[proof of Proposition 4]{Csajbok_Marino_Pepe}.\\
Let us set some notations that we will use throughout this section. 
If $F(\textbf{x})$ is a multivariate polynomial as in \eqref{multivariatepolynomial}, we will denote by
\begin{itemize}
    \item [$a)$] $m_j:=\deg_{q^s} f_j(x_j)= \max \{ i \in \{1,2,\ldots,n-1\} \, | \, a_{ij} \neq 0 \}$, i.e., the $q^s$-\textit{degree} of $f_j(x_j)$, and 
    \item [$b)$] $I_j : =\mathrm{supp}_{q^s}f_j(x_j)=\{i \in \{1,2,\ldots,n-1\} : a_{ij} \neq 0 \}$, i.e., the $q^s$-\textit{support} of $f_j(x_j)$.
\end{itemize}

Let $\Sigma^*=\PG(n(r-1)-1,q^n)$ be a projective space and consider the canonical subgeometry of $\Sigma^*$ defined as follows:
\begin{equation}\label{Sigma}
\boldsymbol{\Sigma}= \{ \langle (\textbf{ x},\textbf{ x}^q,\ldots,\textbf{ x}^{q^{n-1}})\rangle_{\F_{q^n}}\, \colon \, \textbf{ x} \in \F_{q^n}^{r-1}, \textbf{ x} \ne \textbf{0}\} \cong \PG(n(r-1)-1,q), 
\end{equation}
where ${\bf x}^{q^m}=(x_0^{q^m},x_1^{q^m},\ldots,x_{r-2}^{q^m})$ for each $m\in\{0,1,\ldots,n-1\}$. Clearly, $\bf{\Sigma}$ is the set of the points fixed by the collineation $\boldsymbol{\sigma}$ of $\Sigma^*$ of order $n$, defined as
\begin{align}\label{eq:order_n_collineation}
    \boldsymbol{\sigma} \,:\, \langle (x_{00},x_{10},\ldots,x_{r-2,0},x_{01},x_{11},\ldots,x_{r-2,1},\ldots,x_{0,n-1},x_{1,n-1},\ldots,x_{r-2,n-1}) \rangle_{\F_{q^n}} \mapsto \nonumber \\
  \langle  (x^q_{0,n-1},x^q_{1,n-1},\ldots,x^q_{r-2,n-1},x^q_{00},x^q_{10},\ldots,x^q_{r-2,0},\ldots,x^q_{0,n-2},x^q_{1,n-2},\ldots,x^q_{r-2,n-2}) \rangle_{\F_{q^n}}.
\end{align}

\begin{theorem}    \label{Fdirect} Let $L_F$ be a linear set of rank $n(r-1)$ of $\PG(r-1,q^n)$ with $F(\bf{x})$ 
as in \eqref{multivariatepolynomial}. Let $\Sigma^*=\PG(n(r-1)-1,q^n)$. Then, there exist 
\begin{itemize}
    \item [$1.$] a canonical subgeometry $\Sigma$, a generator $\sigma$ of the subgroup of $\PGaL(n(r-1),q^n)$ fixing pointwise  $\Sigma$ and $r-1$ imaginary  points $P_0,P_1,\ldots,P_{r-2}$ of $\Sigma^*$ (wrt $  \Sigma $),
    \item [$2.$] an $(n(r-1)-(r+1))$-dimensional subspace $\Gamma$ of $\Sigma^*$ fulfilling
\begin{itemize}
\item [$i)$]   $P_k^{\sigma^{j_k}} \in \Gamma$ for any $j_k \not \in I_k$, $j_k \not=0$ and $k \in \{0,1,\ldots,r-2\}$,
\item [$ii)$] any line 
 $$\langle P_\ell^{\sigma^{j_\ell}},P_{\ell}^{\sigma^{m_\ell}} \rangle $$
with $j_\ell \in I_\ell  \setminus \{m_\ell\}$ for $\ell \in \{0,1,\ldots,r-2\}$ meets $\Gamma$,
\item [$iii)$] if $r >2$, any line 
$$\langle P_i^{\sigma^{m_i}},P_{r-2}^{\sigma^{m_{r-2}}} \rangle $$
with $i \in \{0,1,\ldots,r-3\}$ meets $\Gamma$,
\end{itemize}
\item [$3.$] an $(r-1)$-dimensional subspace $\Lambda$ of $\Sigma^*$  through the points $P_i$, $i \in \{0,1,\ldots,r-2\}$,
\end{itemize}
such that
$\Gamma \cap  \Sigma = \emptyset = \Gamma \cap \Lambda $ and $\mathrm{p}_{\Gamma,\Lambda}( \Sigma)$ is equivalent to $L_F$.  
\end{theorem}
\begin{proof}
 Let $F(\textbf{x})$ be the multivariate polynomial with coefficients over $\F_{q^n}$ defined as in \eqref{multivariatepolynomial} with $\gcd(s,n)=1$. Let us assume that $\Sigma=\boldsymbol{\Sigma}$ and hence that $\sigma=\boldsymbol{\sigma}^s$ ( as defined  in \eqref{Sigma} and \eqref{eq:order_n_collineation}, respectively) and let  $\Gamma$ 
 be the $(n(r-1)-(r+1))$-dimensional subspace of $\Sigma^*$ with equations  
\begin{equation}\label{Gamma1}
\Gamma : \begin{cases} 
        x_{00} = x_{10} = \cdots = x_{r-2,0} = 0\\
       \sum_{i \in I_0}a_{0i}x_{0,is} + \sum_{i\in I_1}a_{1i}x_{1,is} + \cdots + \sum_{i \in I_{r-2}}a_{r-2,i}x_{r-2,is}=0
       \end{cases}
\end{equation}
and $P_k$, $k \in \{0,1,\ldots,r-2\}$, be the points having $k$-th coordinate equals to $1$ and all the others equal to zero. Clearly, each of these points is imaginary with respect to $ \Sigma$. Moreover, it is straightforward to see that 
\begin{itemize}
\item [$i)$] for any $j_k \notin I_k$ and $j_k \not= 0$, $P_{k}^{\sigma^{j_k}}\in \Gamma$. Indeed, observe that $P_{k}^{\sigma^{j_k}}$ has the $(s j_k(r-1)+k+1)$-th  $(\textrm{mod}\, n(r-1) \,)$ coordinate $1$ and all others equal to zero; 
\item [$ii)$] any line 
$$\langle P_\ell^{\sigma^{j_\ell}},P_{\ell}^{\sigma^{m_\ell}} \rangle $$
with $j_\ell \in I_\ell \setminus \{m_\ell\}$ for $\ell \in \{0,1,\ldots,r-2\}$ meets $\Gamma$ in the point
$Q_{j_\ell,m_\ell}$, with all zero coordinates except the $(sj_{\ell}(r-1)+\ell+1)$-th  $(\textrm{mod}\, n(r-1) \,)$ and the $(sm_\ell(r-1)+\ell+1)$-th $(\textrm{mod}\, n(r-1) \,)$ ones, which are equal to $a_{\ell,m_\ell}$ and $-a_{\ell,j_\ell}$, respectively; and
\item [$iii)$] if $r>2$, any line 
$$\langle P_i^{\sigma^{m_i}},P_{r-2}^{\sigma^{m_{r-2}}} \rangle $$
with $i \in \{0,1,\ldots,r-3\}$ meets $\Gamma$ in the point
$R_{i}$ with all zero coordinates except the $(sm_{i}(r-1)+i+1)$-th $(\textrm{mod}\, n(r-1) \,)$ and the $(sm_{r-2}(r-1)+r-1)$-th $(\textrm{mod}\, n(r-1) \,)$  ones, which are equal to $a_{r-2,m_{r-2}}$ and $-a_{i,m_i}$, respectively.
\end{itemize}
Finally, let $\Lambda=\langle P_0,P_1,\ldots,P_{r-2},P_{r-2}^{\sigma^{m_{r-2}}} \rangle$. Easy computations show that $\Lambda$ is an $(r-1)$-dimensional subspace of $\Sigma^*$ such that $\Gamma \cap  \Sigma = \emptyset = \Lambda  \cap \Gamma $ and $ \mathrm{p}_{\Gamma,\Lambda}(\Sigma)$ is equivalent to $L_F$. 
\end{proof}

Using the notation introduced in Theorem \ref{Fdirect} and taking into account that the points $P_i^{\sigma^j}$'s are independent, it is straightforward to check that the points
\begin{itemize}
\item [$(a)$] $P_{k}^{\sigma^{j_k}}, j_k \not \in I_k, j_k \not =0\, \textnormal{and}\, k \in\{0,1,\ldots,r-2\},$
\item [$(b)$] $Q_{j_\ell,m_\ell} \in \langle P_{\ell}^{\sigma^{j_\ell}}, P_{\ell}^{\sigma^{m_\ell}} \rangle,  j_\ell \in I_\ell \setminus \{m_\ell \}$  and $\ell \in \{0,1,\ldots,r-2\}$, 
\item [$(c)$] $R_i \in \langle P^{\sigma^{m_i}}_i,P^{\sigma^{m_{r-2}}}_{r-2} \rangle, i \in \{0,1,\ldots,r-3\}$ for $r>2$,
\end{itemize}
are independent and they span the subspace $\Gamma$. Next, we prove the following result.

\begin{lemma}\label{moveLPi}
Let $\Sigma$ be a subgeometry of $\Sigma^*=\PG(n(r-1)-1,q^n)$ and let $\sigma$ denote\textcolor{blue}{s} a collineation of order $n$ of $\Sigma^*$ such that $\mathrm{Fix}(\sigma)=\Sigma$. Let $P_0,P_1,\ldots,P_{r-2}$ and $P'_0,P'_1,\ldots,P'_{r-2}$ be imaginary points (wrt $\Sigma$) such that 
\begin{equation*}
\mathcal{L}_{P_i}=\mathcal{L}_{P'_i} \,\,\, \textrm{and} \,\,\, \langle \cL_{P_0},\cL_{P_1},\ldots,\cL_{P_{r-2}}\rangle= \Sigma^*.
\end{equation*}
Then, there exists a collineation $\varphi \in \mathrm{LAut}(\Sigma)$ such that $P^{\varphi}_i=P'_i$ for $i \in \{0,1,\ldots,r-2\}$.    
\end{lemma}

\begin{proof}
Set $\mathcal{L}_i:=\cL_{P_i}=\cL_{P'_i}$. By hypothesis, for each $i  \in  \{0,1,\ldots,r-2   \}$,
$$\Sigma_i:=\mathcal{L}_i \cap \Sigma = \PG(V_i)  \cong\PG(n-1,q)$$ 
and $\mathcal{L}_i=\PG(W_i)\cong \PG(n-1,q^n)$ where $W_i= \langle V_i\rangle_{\F_{q^n}}$. Moreover, note that $\mathrm{Fix}(\sigma|_{{\mathcal{L}_i}})=\Sigma_i$. By applying \cite[Proposition 3.1]{Bonoli-Polverino}, we have that there exists a collineation $\varphi_i \in \mathrm{LAut}(\Sigma_i)$ of $\cL_i$ fixing $\Sigma_i$\textcolor{red}{,} and mapping $P_i$ into $P'_i$. In the remaining part of the proof we will use the same symbol $\varphi_i$ to denote both the collineation and its underlying linear map; i.e., we will write $\varphi_i: W_i \longrightarrow W_i$ for any $i \in \{0,1,\ldots,r-2\}$.  

Since $\Sigma^* =\PG(\oplus_{i=0}^{r-2} W_i)$ and $\Sigma= \PG(\oplus_{i=0}^{r-2} V_i)$, we may
consider the map $\varphi \in \mathrm{PGL}(n(r-1),q^n)$ with  the following associated linear map  
\begin{equation*} w_0+ w_1+w_2+\ldots + w_{r-2} \in \oplus_{i=0}^{r-2} W_i \, \longrightarrow \,  \varphi_0(w_0)+\varphi_1(w_1)+\varphi_2(w_2)+\ldots + \varphi_{r-2}(w_{r-2}) \in \oplus_{i=0}^{r-2} W_i.
\end{equation*}
First, we prove that $\varphi \in \mathrm{LAut}(\Sigma)$, i.e., $\sigma \varphi=\varphi \sigma$. Let $P= \langle w \rangle_{\F_{q^n}}$ with $w= w_0+ w_1+\ldots + w_{r-2}$ where $w_i \in W_i$, then 
\begin{equation*}
\begin{split}
   \sigma(\varphi({\langle w \rangle_{\F_{q^n}}}))&=\sigma(\langle \varphi_0(w_0)+\varphi_1(w_1)+\varphi_2(w_2)+\ldots + \varphi_{r-2}(w_{r-2}) \rangle_{\F_{q^n}})\\
&=\langle\sigma(\varphi_0(w_0))+\sigma(\varphi_1(w_1))+\sigma(\varphi_2(w_2))+\ldots+\sigma(\varphi_{r-2}(w_{r-2}))\rangle_{\F_{q^n}}\\
&=\langle\sigma_{|_{W_0}}(\varphi_0(w_0))+\sigma_{|_{W_1}}(\varphi_1(w_1))+\sigma_{|_{W_{2}}}(\varphi_{2}(w_{2}))+\ldots+\sigma_{|_{W_{r-2}}}(\varphi_{r-2}(w_{r-2}))\rangle_{\F_{q^n}}\\
&=\langle \varphi_0(\sigma(w_0))+\varphi_1(\sigma(w_1))+\varphi_{2}(\sigma(w_{2}))+\ldots+\varphi_{r-2}(\sigma(w_{r-2}))\rangle_{\F_{q^n}}\\
&=\varphi(\langle \sigma(w_0)+\sigma(w_1)+\sigma(w_2)+\ldots+\sigma(w_{r-2})\rangle_{\F_{q^n}})\\
&=\varphi(\sigma(\langle w \rangle _{\F_{q^n}}))
\end{split}
\end{equation*}
where again we have denoted the semilinear map  associated with $\sigma$ and $\varphi$ by the same letters, respectively.
Finally, it is routine to check that $P_i^\varphi=P'_i$, $i \in \{0,1,\ldots,r-2\}$. Thus, the result follows.
\end{proof}
We now prove a generalization of \cite[Proposition 3.1]{Bonoli-Polverino} which will play an important role in the proof of Theorem \ref{GPT}.  
\begin{proposition}\label{prop:transitivity}
Let $\Sigma$ be a subgeometry of $\Sigma^*=\PG(n(r-1)-1,q^n)$. Let $\sigma$ denote a collineation of order $n$ of $\Sigma^*$ such
that $\mathrm{Fix}(\sigma)=\Sigma$. Then, the group $\mathrm{LAut(\Sigma)}$ acts $(r-1)$-transitively on sets of $r-1$ independent points $P_0,P_1,\ldots,P_{r-2}$ such that $\langle \mathcal{L}_{P_0}, \mathcal{L}_{P_1}, \dots, \mathcal{L}_{P_{r-2}}\rangle=\Sigma^* $.
\end{proposition}
\begin{proof}

Let $P_0, P_1,\dots, P_{r-2}$ be the independent points such that $\langle \mathcal{L}_{P_0}, \mathcal{L}_{P_1}, \dots, \mathcal{L}_{P_{r-2}}\rangle=\Sigma^* $. Therefore, any point $P_i$, $i \in \{0,1,\ldots,r-2\}$, is an imaginary point of $\Sigma^*$ (wrt $\Sigma$). So, for any $i \in \{0,1,\ldots,r-2\}$, there exist $n$ points in $\Sigma \cap \mathcal{L}_{P_i}$, say $P_{i0}, P_{i1}, \dots, P_{i(n-1)}$, such that:  $$\{P_{ij}\in \Sigma : i\in\{0,1,\dots,r-2\} \text{ and }j\in\{0,1,\dots,n-1\}\}$$ are in general position in $\Sigma$. Let $Q_0, Q_1,\dots, Q_{r-2}$ be $(r-1)$ points such that $\langle \mathcal{L}_{Q_i} \colon i=0,1,\ldots,r-2 \rangle=\Sigma^* $. So, in the same way, we get a set $$\{Q_{ij}\in \Sigma : i\in\{0,1,\dots,r-2\} \text{ and }j\in\{0,1,\dots,n-1\}\}$$ of points of $\Sigma$ in general position, as well. 
Since  $\mathrm{LAut}(\Sigma)$ acts transitively on subset of $\Sigma$ of order $n(r-1)$ which are in general position, there exists a collineation $\psi\in\mathrm{LAut}(\Sigma)$ such that $P_{ij}^{\psi}=Q_{ij}$ for any $i \in \{0,1,\ldots,r-2\}$ and $j \in \{0,1,\ldots,n-1\}$. 
Moreover, it holds that 
\begin{equation*}
\begin{split}
    \cL_{P_i}^{\psi}&=  ( \langle P_{i0},P_{i1},\ldots, P_{i(n-1)}\rangle)^{\psi}=\langle P_{i0}^\psi,P_{i1}^\psi,\ldots, P_{i(n-1)}^\psi \rangle\\
    &=\langle Q_{i0},Q_{i1},\ldots, Q_{i(n-1)} \rangle = \cL_{Q_i}.
    \end{split}
\end{equation*}
Now, if $P_i^{\psi}=Q_i$, we get the result. Otherwise, assume $P_i^\psi=R_i$, for some $R_0,R_1,\ldots,R_{r-2}$. Since $\mathcal{L}_{R_i}=\mathcal{L}_{Q_i}$, by Lemma \ref{moveLPi} there exists $\varphi \in \mathrm{LAut}(\Sigma)$ such that $R_i^{\varphi}=Q_i$ and so $P_i^{\psi \varphi}:=\varphi(\psi(P_i))=Q_i$. This completes the proof. 
\end{proof}

We are now in the position to prove the following inverse of Theorem \ref{Fdirect}.
\begin{theorem}
    {\label{GPT}}
   Let $r,n \geq 2$, $I_j \subseteq \{1,\ldots,n-1\},$ with $j \in \{0,1,\ldots,r-2\}$.
   Let $\Sigma$ be a canonical subgeometry of $\Sigma^*=\PG(n(r-1)-1,q^n)$ and consider $\Gamma$ and $\Lambda$ subspaces of $\Sigma^*$ with dimensions $(n(r-1)-(r+1))$ and $(r-1)$, respectively, such that $\Gamma \cap \Sigma = \emptyset = \Lambda  \cap \Gamma $.\\
    Assume $L= \mathrm{p}_{\Gamma,\Lambda}(\Sigma)$ is an $(r-1,n(r-1)-2)_q$-evasive linear set of rank $n(r-1)$.\\
    If there exist $r-1$ points $P_0, P_1,\ldots, P_{r-2}$ such that $\Gamma$ is spanned by points defined as in $(a)$, $(b)$ and $(c)$ where $\sigma$ is a generator of the subgroup of $ \PGaL(n(r-1),q^n)$ fixing $\Sigma$ pointwise.
    Then, 
     \begin{enumerate}
         \item $P_0, P_1,\ldots, P_{r-2}$ are imaginary points (wrt $\Sigma$), and  
         \item $L \cong L_F$ as in (\ref{linearset}) where $\mathrm{supp }_{q^s}f_j(x_j)=I_j$, for some integer $1 \leq s \leq n-1$  with $\gcd(s,n)=1$ and $j=0 ,\ldots,r-2$.
     \end{enumerate}
\end{theorem}

\begin{proof}
Proving point {\it 1.} of the statement, is equivalent to show that $\dim{\cal L}=n(r-1)-1$, where $\mathcal{L}=\langle \mathcal{L}_{P_0},\mathcal{L}_{P_1},\ldots,\mathcal{L}_{P_{r-2}} \rangle$.
By our hypothesis, $\Gamma\subset {\cal L}$ and hence $$n(r-1)-(r+1)<\dim (\mathcal{L})\leq n(r-1)-1.$$
Assume that $\dim {\cal L} < n(r-1)-1$. Precisely, let
$\dim \mathcal{L}=n(r-1)-(r+1)+t$, where $t\in\{1,2,\dots,r-1\}$. Then, by Grassman's Formula  $\dim(\Lambda \cap \mathcal{L}) = t-1.$ Since $\mathcal{L}$ is fixed by $\sigma$, we get $$ \mathrm{wt}_L(\Lambda \cap \mathcal{L}) = \mathrm{rk} (\mathcal{L} \cap \Sigma)=\dim  \mathcal{L}+1=n(r-1)-(r+1)+t+1.$$ This means that we have a $t$-dimensional $\mathbb{F}_{q^n}$-subspace of $\F_{q^n}^r$, which meets an  $\F_q$-linear subspace in an $\F_q$-linear subspace with dimension $(n(r-1)-(r+1)+t+1)$ over $\F_q$. Since by \cite[Proposition 2.6]{Bartoli_Csajbok_Marino_Trombetti} an $(r-1,n(r-1)-2)_q$-evasive is also $(t, n(r-1)-(r+1)+t)_q$-evasive, we contradict the assumption. Therefore, $\dim(\mathcal{L})=n(r-1)-1$. By Proposition \ref{prop:transitivity}, let $\Sigma$ as in \eqref{Sigma}, then there exists an integer $ 1 \leq s \leq n-1$ with $\gcd(s,n)=1$ such that $\sigma=\boldsymbol{\sigma}^s$. Then, we may assume that any $P_k$, $k=0,1,\ldots,r-2$, is a point with homogeneous coordinates equal to 1 in the $k$-th  position and 0 in all others. Furthermore, since $\Gamma$ is spanned by points defined as in $(a)$, $(b)$ and $(c)$, it turns out to have equations
\begin{equation*}
\Gamma : \begin{cases} 
        x_{00} = x_{10} = \cdots = x_{r-2,0} = 0\\
       \sum_{i \in I_0}a_{0i}x_{0,is} + \sum_{i\in I_1}a_{1i}x_{1,is} + \cdots + \sum_{i \in I_{r-2}}a_{r-2,i}x_{r-2,is}=0
       \end{cases}
\end{equation*}
 with some $a_{ji} \in \F^*_{q^n}$, $j=0,\ldots,r-2$ and $i \in I_j$. Consequently, $L \cong L_F$, where  $F(\textbf{x})=\sum_{j=0}^{r-2}\sum_{i\in I_j}a_{ji}x^{q^{si}}$.
\end{proof}

\section{Maximum scattered linear sets of $\PG(1,q^n)$}

In this section, we will focus on maximum scattered linear sets of the projective line $\PG(1,q^n)$. These are covered as a special case in the results of the previous section providing $r=2$. As expressed in Section \ref{sec: intro}, our aim is to characterize the families of these linear sets which were introduced from 2018 onwards (see \cite{Csajbok_Marino_Zullo, Csajbok-Marino-Polverino-Zanella, longobardi_marino_trombetti_zhou, longobardi_zanella, neri_santonastaso_zullo, SmaZaZu}) by means of certain properties of their projection vertices.

For this purpose, we will start by recalling some characterisation results obtained in this direction, regarding the two classical and up to date most investigated examples: pseudoregulus and Lunardon-Polverino type scattered linear sets of the line.

Let $\ell$ be an integer such that $\ell \mid n$. In the following, we will indicate by 
\begin{equation*}
   \mathrm{N}_{q^n/q^\ell}(x)=x^{\frac{q^{n}-1}{q^\ell-1}} \quad \textnormal{and} 
\quad \mathrm{Tr}_{q^n/q^\ell}(x)=\sum_{i=0}^{n/\ell-1}  x^{q^{i \ell }}
, 
\end{equation*} the \textit{norm} and \textit{trace} of $\F_{q^n}$ over $\F_{q^\ell}$ of an element $x \in \F_{q^n}$, respectively. 

As noted in \cite{Csajbok_Zanella}, two $\F_q$-subspaces of $\F_{q^n}^2$ defining the same scattered linear set in $\PG(1, q^n)$ are not necessarily on the same orbit under the action of $\GaL(2, q^n)$.  This leads in \cite{classes} to the introduction of the notion of $\mathcal{Z}(\GaL)$-class and $\GaL$-class of a linear set of rank $n$ in $\PG(1,q^n)$. Precisely,
\begin{definition}\cite[Definition 2.4]{classes}
Let $L_U$ be an $\F_q$-linear set of $\PG(1, q^n)$ of rank $n$ with maximum field of linearity $\F_q$, i.e., $L_U$ is not an $\F_{q^u}$-linear set
for any $u > 1$. Then $L_U$ has $\mathcal{Z}(\GaL)$-\textit{class} $z$, if $z$ is the largest integer
such that there exist $\F_q$-subspaces $U_1, U_2,\ldots, U_z$ of $\F_{q^n}^2$ with 
\begin{enumerate}
    \item $L_{U_i} = L_U$ for
$i \in \{1, 2, \ldots, z\}$, and
\item there is no $\lambda 
 \in \F_{q^n}^*$ such that $U_i= \lambda U_j$  for each
$i \neq  j$ and  $i, j \in \{1, 2, . . . , z\}$.
\end{enumerate} 
\end{definition}
 Similarly, the notion of $\GaL$-class of an $\F_q$-linear set was defined.

\begin{definition} \cite[Definition 2.5]{classes}. Let $L_U$ be an $\F_q$-linear set of $\PG(1, q^n)$ of rank $n$ with maximum field of linearity $\F_q$. Then $L_U$ is of $\GaL$-\textit{class} $c$, if $c$ is the largest integer
such that there exist $\F_q$-subspaces $U_1,U_2,\ldots,U_c$ of $\F_{q^n}^2$ with
\begin{enumerate}
    \item  $L_{U_i} = L_U$ for $i\in \{1, 2, \ldots, c\}$, and
    \item there is no $\kappa  \in  \GaL(2, q^n)$ such that $U_i = U_j^\kappa$ for each $i \neq j$, $i, j \in \{1, 2,\ldots, c\}$.
\end{enumerate} 
\end{definition}

In the following, we will denote the $\mathcal{Z}(\GaL)$-class and the $\GaL$-class of a linear set $L=L_U$ by $c_{\mathcal{Z}(\GaL)}(L)$ and $c_{\GaL}(L)$, respectively. These integers are invariant under the action of the group $\PGaL(2,q^n)$ on $L$, see \cite[Proposition 2.6]{classes}, and it is clear that $c_{\GaL}(L)  \leq c_{\mathcal{Z}(\GaL)}(L)$. The linear set $L$ is called $simple$ if its  $\GaL$-class is  one. Also, in \cite{Zanella-Zullo}, the $\GaL$-class of a linear set was rephrased via a group action property on its projection vertices, see \cite[Section 5.2]{classes}, \cite[Theorem 6 and 7]{Csajbok_Zanella} and \cite[Theorem 3.4]{Zanella-Zullo}. More precisely, the following was proved

 \begin{theorem}  \label{geoclass} 
 Let $\Sigma$ be a canonical subgeometry and  $\Lambda$  be a line of $\Sigma^*=\PG(n-1,q^n)$. The  $\GaL$-class of a linear set $L \subset \Lambda$ is the number of orbits, under the action of $\Aut(\Sigma)$ of $(n - 3)$-subspaces $\Gamma$ of $\Sigma^*$ disjoint from $\Sigma$ and from the line $\Lambda$ such that $\mathrm{p}_{\Gamma,\Lambda}(\Sigma)$ is equivalent to $L$.
\end{theorem}

\medskip

Let $L_U$ be a linear set of $\PG(1, q^n)$ and let $\tau$ be a  polarity of the projective line. Then, it is
always possible to construct another linear set, which is called the dual linear set of $L_U$ with
respect to the polarity $\tau$, see \cite{Polverino}. In particular, if $\beta : \F_{q^n}^2\times \F_{q^n}^2 \longrightarrow  \F_{q^n}$ is a non-degenerate reflexive sesquilinear
form on $\F_{q^n}^2$ determining a polarity $\tau$, then the map $\beta'=\mathrm{Tr}_{q^n/q} \circ \beta$ is a non-degenerate reflexive sesquilinear form on $\F_{q^n}^2$, when this is regarded
as a $2n$-dimensional vector space over $\F_q$. 

Let $\perp_\beta$ and $\perp_{\beta'}$ be the orthogonal complement maps defined by $\beta$ and $\beta'$ on the lattices of the $\F_{q^n}$-subspaces and $\F_q$-subspaces of $\F_{q^n}^2$, respectively. The \textit{dual linear set $L_U^\tau$ of $L_U$ with respect to the polarity $\tau$} is the $\F_q$-linear set of rank $n$ of $\PG(1, q^n )$ defined by the orthogonal complement $U^{\perp_{\beta'}}$ and, up to
projective equivalence, such a linear set does not depend on $\tau$, see \cite[Proposition 2.5]{Polverino}.\\

As a direct consequence of \cite[Theorem 1]{Pepe}, we may state the following result:

\begin{theorem}\label{thm:Gamma_L_class}
Let $L_f$ be a maximum scattered linear set of $\PG(1,q^n)$. If $W$ is an $n$-dimensional $\F_q$-subspace of $\F_{q^n}^2$ such that $L_{W}=L_f$ with $f(x)=\sum_{i=1}^{n-1}a_ix^{q^{i}}$, then one of the following possibilities may occur:
\begin{enumerate}
 \item $L_f$ is equivalent to a linear set of \textit{pseudoregulus type} (see \cite{Blokhuis-Lavrauw}), i.e., 
\begin{equation}\label{blokhuis_lavrauw1}
 L^{1,n}_{s}=\{\langle (x, x^{q^s}) \rangle_{\F_{q^n}} \, \colon \, x \in \F^*_{q^n} \};
 \end{equation}
\item $W= \lambda U_f$ for some $\lambda \in \F^*_{q^n}$;
\item $W= \lambda U_{\hat{f}}$ for some $\lambda \in \F_{q^n}^*$, where \begin{equation*}
    \hat{f}(x)=\sum_{i=1}^{n-1}a_{n-i}^{q^{i}}x^{q^i}.
\end{equation*}
\end{enumerate}  
In particular, if $L_f$ is of pseudoregulus type $c_{\GaL}(L_f)= \phi(n)/2$, where $\phi$ is the totient function, otherwise $c_{\GaL}(L_f) \leq 2$.
\end{theorem}

 \noindent In \cite{Lunardon_Polverino1}, Lunardon and Polverino presented a class of linear sets which was generalised later by Sheekey in \cite{sheekey}. They are known in the literature  as linear sets of \textit{LP}-\textit{type} and an element of this family is $\PGaL$-equivalent to 
\begin{equation}\label{Lunardon_Polverino1}
L^{2,n}_{s,\eta}= \{ \langle (x, \eta x^{q^{s}}+x^{q^{n-s}})\rangle_{\F_{q^n}} \, \colon \, x \in \F^*_{q^n}\},
\end{equation}  
where $\mathrm{N}_{{q^n}/q}(\eta) \notin \{ 0,1\}$, $\gcd(r,n)=1$ for any $n\geq 3$ and $q>2$.

In the following, we will state two theorems containing characterization results for the linear sets of pseudoregulus and LP-type by means of some geometric properties of their vertices.

\begin{theorem} \cite[Theorem 2.3]{Csajbok_Zanella1} \label{pseudchar}
Let $\Sigma$ be a canonical subgeometry of $\PG(n-1, q^n)$, $q > 2$, $n \geq 3$. Assume that $\Gamma$ and $\Lambda$ are $(n-3)$-subspace and a line of $\PG(n-1, q^n)$, respectively, such that $\Sigma\cap\Gamma=\Lambda\cap\Gamma= \emptyset$. Then the following assertions are equivalent:
\begin{enumerate} 
    \item  The set $\mathrm{p}_{\Gamma, \Lambda}(\Sigma)$ is a scattered $\mathbb{F}_q$-linear set of pseudoregulus type;
    \item  there exists a generator $\sigma$ of the subgroup of $\PGaL(n,q^n)$ fixing $\Sigma$ pointwise such that $\dim(\Gamma\cap\Gamma^{\sigma})=n-4$; furthermore $\Gamma$ is not contained in the span of any hyperplane of $\Sigma$;
     \item There exists a point $P$ and a generator $\sigma$ of the subgroup of $\PGaL(n, q^n)$ fixing $\Sigma$ pointwise, such that $\langle P , P^{\sigma}, \ldots, P^{\sigma^{n-1}}\rangle= \PG(n-1, q^n)$, and
$$\Gamma=\langle P ,P^{\sigma},\ldots, P^{\sigma^{n-3}}\rangle.$$
\end{enumerate}

\end{theorem}

\noindent In the same spirit, later in \cite{Zanella-Zullo}, the authors characterized linear sets of LP-type. In order to do this, they introduced the notion of intersection number of a subspace with respect to a collineation fixing pointwise a canonical subgeometry. More precisely, let $\Gamma$ be a non-empty $k$-dimensional subspace of $\PG(n-1,q^n)$, the \textit{intersection number of $\Gamma$ with respect to $\sigma$}, denoted by $\ints$, is defined as the least positive  integer $\gamma$ satisfying 
\begin{equation*}
    \dim(\Gamma \cap \Gamma^\sigma \cap \ldots \cap \Gamma^{\sigma^\gamma}) > k -2 \gamma.
\end{equation*}

It is easy to see that $\ints$ is invariant under the action of $\Aut(\Sigma)$. By Theorem \ref{pseudchar}, a maximum scattered linear set is of pseudoregulus type if and only if $\ints=1$ for some $\sigma$ fixing $\Sigma$ pointwise. Also in \cite{Zanella-Zullo} the following result was proven.

\begin{theorem}\cite[Theorem 3.2]{Zanella-Zullo}\label{LPcharct}
    Let $\Sigma$ be a canonical subgeometry of $\PG(n-1,q^n)$, $q>2$ and $n\geq4$. Let $L$ be a scattered linear set in  a line $\Lambda$ of $\PG(n-1,q^n)$. Then $L$ is a linear set of LP-type  if and only if 
    \begin{enumerate}
        \item [\textit{1.}] there exists an $(n-3)$-subspace $\Gamma$ of $\PG(n-1,q^n)$ such that  $\Gamma\cap\Sigma=\Gamma\cap\Lambda=\emptyset$ and $L=\mathrm{p}_{\Gamma,\Lambda}(\Sigma)$;
        \item [\textit{2.}] there exists a generator $\sigma$ of the subgroup of $\PGaL(n,q^n)$ fixing $\Sigma$ pointwise such that $\ints=2$;
        \item [\textit{3.}] there exists a unique point $P\in \PG(n-1,q^n)$ and some point $Q$ such that $$\Gamma=\langle P, P^{\sigma},\dots, P^{\sigma^{n-4}},Q\rangle ;$$
        \item [\textit{4.}] the line $\langle P^{\sigma^{n-1}}, P^{\sigma^{n-3}}\rangle$ meets $\Gamma$.
    \end{enumerate}
\end{theorem}

Moreover, as a consequence of Theorem \ref{thm:Gamma_L_class} we get a slight strengthening of \cite[Theorem 3.5]{Zanella-Zullo}, indeed removing the assumptions $n \leq 6$ and $n=8$. Precisely, we may state the following result.

\begin{theorem} Let $L$ be a maximum scattered linear set in a line $\Lambda$ of $\PG(n-1,q^n)$ for any $n \geq 4$ and $q>2$. Then $L$ is a linear set of LP-type if and only if for each $(n - 3)$-subspace $\Gamma$ of $\PG(n - 1, q^n)$ such that $L = \mathrm{p}_{\Gamma ,\Lambda}(\Sigma)$, the following holds:
\begin{itemize}
\item [$i)$] there exists a generator $\sigma$ of the subgroup of $\PGaL(n, q^n)$ fixing $\Sigma$ pointwise such that $\ints= 2$;
\item [$(ii)$] if $P$ is the unique point of $\PG(n -1, q^n)$ such that
\begin{equation*}
    \Gamma = \langle P, P^\sigma,\ldots, P^{\sigma^{n-4}},Q \rangle
\end{equation*}
then the line $\langle P^{\sigma^{n-1}}, P^{\sigma^{n-3}}\rangle$ meets $\Gamma$.
\end{itemize}
\end{theorem}

\medskip
\noindent In the last decade, many new constructions of maximum scattered linear sets of the projective line were exhibited. In the following, in the light of what was done for the examples belonging to classes \eqref{blokhuis_lavrauw1} and \eqref{Lunardon_Polverino1}, we will provide characterization results similar to Theorems \ref{pseudchar} and \ref{LPcharct},  for all the other constructions mentioned above. 

Toward this aim we first premise a complete list of examples that have appeared from 2018 onward.

 When $n\in \{6,8\}$,  in \cite{Csajbok-Marino-Polverino-Zanella}, the following family of $\F_q$-linear sets was presented 

\begin{equation}\label{ex:CMPZ}
 L^{3,n}_{\ell,\eta}= \{ \langle (x, \eta x^{q^{\ell}}+x^{q^{\ell
+n/2}})\rangle_{\F_{q^n}} \, \colon \, x \in \F^*_{q^n}\},   
\end{equation}
where $\gcd(\ell,n/2)=1$ and $\eta \in \F_{q^n}$ such that $\N_{q^n/q^{n/2}}(\eta) \not \in \{0,1\}$. Moreover in \cite{Csajbok-Marino-Polverino-Zanella}, it was shown that, for some choices of $q$ and  $\eta$, $L_{\ell,\eta}^{3,n}$ turns out to be scattered, see \cite[Theorem 7.1 and Theorem 7.2]{Csajbok-Marino-Polverino-Zanella}.
Later, necessary and sufficient conditions for these examples in order to be scattered, were also provided. Precisely,
\begin{itemize}
\item[-] if $n=6$ and $q>2$, $L_{2,\eta}^{3,n}$ is scattered if and only if the equation 
\begin{equation*}
Y^2-(\Tr_{q^3/q}(\gamma)-1)Y+\N_{q^3/q}(\gamma)=0, 
\end{equation*}
where $\gamma = -\frac{ \eta^{q^3+1}}{1-\eta^{q^3+1}}$,
admits two solutions in $\F_q,$ see \cite{BCM} and \cite[Theorem 7.3]{polzullo}, 
\item[-] if $n=8$ and $q \leq 11$ or $q \geq 1039891$ odd, $L_{\ell,\eta}^{3,n}$ is scattered if and only $\mathrm{N}_{q^8 / q^4}(\eta) = -1$, see \cite{Timpanella_Zini}. 
\end{itemize}

In 2018, another example was found in the case when $n=6$ and $q$ is odd, consisting of linear sets with the following shape:

\begin{equation}{\label{tri11}}
   L^{4,6}_{\delta} = \{ \langle (x, x^q + x^{q^3}+ \delta x^{q^5})\rangle_{\F_{q^6}}\, \colon \, x \in \F^*_{q^6}\} , 
\end{equation} where $\delta \in \F_{q^6}^*$ such that $\delta^2+\delta= 1;$ see \cite{Csajbok_Marino_Zullo} and \cite{Marino_Montanucci_Zullo}. Very recently, it has been proven that for a different set of conditions on $\delta$, $L_{\delta}^{4,6}$ is
scattered also for $q$ even and large enough, see \cite{BartLongMarTimp}.\\ 
In the series of  articles \cite{Bartoli_Zanella_Zullo, gupta_longobardi_trombetti,longobardi_marino_trombetti_zhou,longobardi_zanella,neri_santonastaso_zullo,SmaZaZu,Zanella-Zullo}, various authors studied the class of linear sets defined as
\begin{equation}\label{quadrinomial}
L_{m,h,s}^{5,t}=\{ \langle(x, \psi_{t,m,h,s}(x))\rangle_{\mathbb{F}_{q^{2t}}} : x \in \mathbb{F}_{q^{2t}}^* \}
\end{equation}
where 
\begin{equation*}
\psi_{t,m,h,s}(x)=x^{q^{s(t-1)}}+h^{1-q^{s(2t-1)}}x^{q^{s(2t-1)}}+m \left (x^{q^s}-h^{1-q^{s(t+1)}}x^{q^{s(t+1)}} \right )
\end{equation*}
for $q$ odd, $t \geq 3$, $\gcd(s,2t)=1$ and $m,h \in \F_{q^{2t}}^*$ with $\mathrm{N}_{q^{2t}/q^t}(h)=-1$.

The linear set $L_{m,h,s}^{5,t}$ turns out to be scattered under one of the following hypotheses:
\begin{enumerate}
    \item [$(i)$] $m=1$ and $h \in \F_{q^t}^*$, see \cite{Bartoli_Zanella_Zullo, longobardi_zanella, neri_santonastaso_zullo,Zanella-Zullo}.
    \item [$(ii)$] $m=1$ and $h \in \F_{q^{2t}}\setminus \F_{q^t}$, see \cite{Bartoli_Zanella_Zullo,gupta_longobardi_trombetti,longobardi_marino_trombetti_zhou,neri_santonastaso_zullo}.
    \item [$(iii)$] $m$ belongs to $\F_{q^t}$ and it is neither a $(q-1)$-th power nor a $(q+1)$-th power of an element belonging to $\ker \mathrm{Tr}_{q^{2t}/q^t}$ and $h=1$, see \cite{SmaZaZu}.
    \end{enumerate} 
Note that if $h \in \F_{q^t}^*$, $t \geq 3$ odd and $\mathrm{N}_{q^{2t}/q^t}(h)=-1$ implies that $q \equiv 1 \pmod 4$.

\noindent Moreover, in \cite{Bartoli_Zanella_Zullo}, the authors showed that $L^{5,3}_{1,h,1}$ does not lie in Classes \eqref{blokhuis_lavrauw1},\eqref{Lunardon_Polverino1},\eqref{ex:CMPZ} and  \eqref{tri11}, except when $q$ is a power of $5$, and $h \in \F_q^*$. In this case, the linear set is equivalent to $L^{4,6}_{\delta}$ for some $\delta \in \F_{q^6}^*$, see \cite[Corollary 3.11]{Bartoli_Zanella_Zullo}. In  \cite{gupta_longobardi_trombetti}, it was proven that $L^{5,4}_{1,h,s}$ is not $\PGaL$-equivalent to any linear sets of $\PG(1,q^8)$ belonging to classes in \eqref{blokhuis_lavrauw1},\eqref{Lunardon_Polverino1},\eqref{ex:CMPZ}, \eqref{tri11}, see \cite[Theorem 4.4]{gupta_longobardi_trombetti}.

\medskip

\noindent 
As we have seen before, in \cite{Csajbok_Zanella1,Zanella-Zullo}, linear sets of pseudoregulus and LP-type were characterized exploiting the intersection number of their vertex of projection.  This approach does not work for scattered linear sets belonging to the other classes described above. Indeed in the following we show that  there exist projection vertices of inequivalent linear sets having the same intersection number.
To this aim, let us consider
\begin{equation*}
    \Sigma=\{\langle(x,x^q, \dots,x^{q^5})\rangle_{\mathbb{F}_{q^6}}:x\in \mathbb{F}_{q^6}^*\},
\end{equation*}
the canonical subgeometry of $\mathrm{PG}(5,q^6)$ fixed by the collineation 
\begin{equation*}
\sigma: \langle( x_0, x_1, \dots, x_5)\rangle_{\mathbb{F}_{q^6}}\in \mathrm{PG}(5,q^6)\longrightarrow\langle(x^q_5, x^q_0, x^q_1, \dots, x^q_4)\rangle_{\mathbb{F}_{q^6}}\in \mathrm{PG}(5,q^6).
\end{equation*}
Let $L^{3,6}_{1,\eta}$ be defined as in \eqref{ex:CMPZ}. It is easy to see that this example is equivalent to a linear set obtained as the projection of $\Sigma$ from the solid $\Gamma_\eta$ to the line $\Lambda_\eta$ having the following equations:  
\begin{align}
\Gamma_\eta:\begin{cases}
    x_0=0\\
    x_1+\eta x_4=0
    \end{cases}   &     &\Lambda_\eta: 
        x_1=x_2=x_3=x_5=0 .
\end{align}
We have $\Gamma_\eta\cap\Sigma=\emptyset=\Gamma_\eta\cap\Lambda_\eta$. Moreover, $\dim(\Gamma_\eta\cap\Gamma_\eta^\sigma)=1$ and $\dim(\Gamma_\eta\cap\Gamma_\eta^\sigma\cap\Gamma_\eta^{\sigma^2})=-1$, hence $\mathrm{intn}_{\sigma}(\Gamma_\eta)=3$.
Similarly, the linear set $L^{5,3}_{1,h,1}$ is equivalent to the projection of $\Sigma$ from the vertex $\Gamma_h$ on the line $\Lambda_{h}$ with equations
\begin{align}
\Gamma_h:\begin{cases}
    x_0=0\\
    h^{q-1}x_1-h^{q^2-1} x_2+x_4+x_5=0
    \end{cases}    &     &\Lambda_h: 
        x_1=x_2=x_3=x_4=0
\end{align}
and as shown in \cite[Section 3.2]{Bartoli_Zanella_Zullo}, $\mathrm{intn}_{\sigma}(\Gamma_{h})=3$ as well.

\medskip
\noindent Since, Theorem \ref{Fdirect} will be crucial in the subsequent part, we will state it when $r=2$, i.e., for linear sets of the projective line $\PG(1,q^n)$.

\begin{proposition}\label{directline} Let $L_f$ be a linear set of rank $n$ of the projective line $\PG(1,q^n)$ as defined in (\ref{multivariatepolynomial}) with $f(x)=\sum_{i=1}^{n-1} a_i x^{q^{si}}$,  $m=\deg_{q^s}f(x)$ and $I=\mathrm{supp}_{q^s}f(x)$.
Then, there exist a canonical subgeometry $\Sigma$ of $\Sigma^*=\PG(n-1,q^n)$ and an imaginary point $P$ of $\Sigma^*$ (wrt $\Sigma$) such that $L_f$ is equivalent to $\mathrm{p}_{\Gamma,\Lambda}(\Sigma)$ with
\begin{equation}\label{gamma}
\Gamma_f= \langle P^{\sigma^j}, Q_i \colon j \not \in I , j \neq 0 \, \textnormal{and}\, i \in I\setminus\{m\}\rangle, 
\end{equation}
where $Q_i \in \langle  P^{\sigma^i}, P^{\sigma^m} \rangle $, $i \in  I \setminus \{m\}$
and $\Lambda= \langle P,P^{\sigma^m}  \rangle$.
\end{proposition}

We observe that in light of Theorems \ref{geoclass}, \ref{thm:Gamma_L_class} and Proposition  \ref{directline}, the result above can be further specialised for maximum scattered linear sets not of pseudoregulus type as follows.

\begin{theorem}\label{thm:embed}
    Let $L_f$ be a maximum scattered linear set of the projective line $\PG(1,q^n)$ not of pseudoregulus type with $f(x)=\sum_{i=1}^{n-1}a_ix^{q^{si}}$, $m=\deg_{q^s}f(x)$ and $I=\mathrm{supp}_{q^s}f(x)$. Then, there exist a canonical subgeometry $\Sigma$ of $\Sigma^*=\PG(n-1,q^n)$ and an imaginary point $P$ of $\Sigma^*$ (wrt $\Sigma)$ such that
    \begin{enumerate}
        \item $L_f$ is equivalent to $\mathrm{p}_{\Gamma_f,\Lambda_f}(\Sigma)$ with 
\begin{equation}\label{gammaf}
\Gamma_f= \langle P^{\sigma^j}, Q_i \colon j \not \in I , j \neq 0 \, \textnormal{and}\, i \in I\setminus\{m\}\rangle \,\, \textnormal{ and } \Lambda_f= \langle P,P^{\sigma^m}  \rangle
\end{equation}
where $Q_i \in \langle  P^{\sigma^i}, P^{\sigma^m} \rangle $ and $i \in  I \setminus \{m\}$,
\item $L_f$ is equivalent to $\mathrm{p}_{\Gamma_{\hat{f}},\Lambda_{\hat{f}}}(\Sigma)$ with
\begin{equation}\label{gammahatf}
\Gamma_{\hat{f}}=\langle P^{\sigma^j}, Q_i \colon j \not \in \hat{I} , j \neq 0 \, \textnormal{and}\, i \in \hat{I}\setminus\{\hat{m}\}\rangle \,\, \textnormal{ and } \Lambda_{\hat{f}}= \langle P,P^{\sigma^{\hat{m}}}  \rangle  \end{equation}
 where $Q_i \in \langle  P^{\sigma^i}, P^{\sigma^{\hat{m}}} \rangle $, $i \in  \hat{I} \setminus \{\hat{m}\}$,  $\hat{I}=\mathrm{supp}_{q^s}\hat{f}(x)$, and $\hat{m}=\deg_{q^s} \hat{f}(x)$.
\end{enumerate}
In particular, if $c_{\GaL}(L_f)=2$ and $\Lambda_f=\Lambda_{\hat{f}}$ then $ \Gamma_{f}$ and $\Gamma_{\hat{f}}$ do not belong to the same orbit under the action of $\Aut(\Sigma)$, otherwise they do. 
\end{theorem}

\begin{remark}\label{aggiunto}
We stress here the fact that in previous statment, the vertex $\Gamma_{\hat f}$ can be spanned by the same $P^{\sigma^j}$'s and the points $T_i \in \Gamma_{\hat f} \cap \langle P^{\sigma^i},P^{\sigma^{\tilde{m}}}\rangle,$ where $i \in \hat I \setminus \{\tilde{m}\}$ and $\tilde{m} = \min\hat{I}$, as well.  
\end{remark}
Putting together Theorem \ref{thm:embed} above with \cite[Theorem 2]{Csajbok_Zanella} and \cite[Theorem 3]{LavrauwVandeVoorde}, we derive the following result.

\begin{theorem}\label{thm:move-gamma}
Let  $\Sigma \cong \PG(n - 1, q)$ be a canonical subgeometry of $\Sigma^*= \PG(n - 1, q^n)$. Let $L$  be a linear set contained in a line  $\Lambda$ of $\Sigma^*$ such that $L=\mathrm{p}_{\Gamma,\Lambda}(\Sigma)$ where $\Gamma$ is an $(n - 3)$-subspace of $\Sigma^*$ such that
$\Gamma \cap \Sigma = \emptyset= \Gamma \cap \Lambda$. Let us consider $L_f$ be a maximum scattered linear set of $\PG(1,q^n)$ not of pseudoregulus type and $\Gamma_f$, $\Gamma_{\hat{f}}$, $\Lambda_{f}$ and $\Lambda_{\hat{f}}$ as defined in \eqref{gammaf} and $\eqref{gammahatf}$. Then, if $L$ is simple  or $\Lambda_f=\Lambda_{\hat{f}}$, 
the following statements are equivalent
\begin{itemize}
\item [$(i)$] $L_f \cong L$;
\item [$(ii)$] there exists a collineation $\beta$ of $\Sigma^*$ such that $\mathbf{\Sigma}^{\beta}=\Sigma$ (cf. \eqref{Sigma}) and either $\Gamma = \Gamma^\beta_f$ or  $\Gamma = \Gamma^\beta_{\hat{f}}$.
\end{itemize}
\end{theorem}
\begin{proof}
$(i) \Rightarrow (ii)$ Let us suppose that  $L_f \cong L$.  If the $\GaL$-class of $L_f$, and hence of $L$, is one then the result follows from \cite[Theorem 6]{Csajbok_Zanella}. Suppose that $c_{\GaL}(L_f)=c_{\GaL}(L)=2$ and $\Lambda_f=\Lambda_{\hat{f}}$. Let us put $\Gamma_1=\Gamma_f$, $\Gamma_2=\Gamma_{\hat{f}}$  and $\boldsymbol{\Lambda}=\Lambda_{f}=\Lambda_{\hat{f}}$ and  consider $L_1=\mathrm{p}_{\Gamma_1,\boldsymbol{\Lambda}}(\boldsymbol{\Sigma})$ and $L_2=\mathrm{p}_{\Gamma_2,\boldsymbol{\Lambda}}(\boldsymbol{\Sigma})$. 
By Theorem \ref{thm:embed}, there exists  $ \alpha_i: \boldsymbol{\Lambda}\longrightarrow \Lambda$ such that $L_i^{\alpha_i}=L$, $i=1,2$. Then  $\alpha_i$ can be extended to an element, which we indicate with the same symbol $\alpha_i$, belonging to  $\PGaL(n, q^n)$  such that $\Gamma_i^{\alpha_i}=\Gamma$, $i=1,2$. Thus, $L=L_i^{\alpha_i}=\mathrm{p}_{\Gamma,{\Lambda}}(\boldsymbol{\Sigma}^{\alpha_i})$, $i=1,2$.\\
In order to get $(ii)$, it is enough to prove the existence of either a map $\phi\in \PGaL(n,q^n)$ such that $\Gamma^{\phi}=\Gamma$, $\boldsymbol{\Sigma}^{\alpha_1 \phi}=\Sigma$ or of a map $\psi \in \PGaL(n,q^n)$ such that $\Gamma^{\psi}=\Gamma$, $\boldsymbol{\Sigma}^{\alpha_2\psi}=\Sigma$. Indeed, if such a $\phi$ or $\psi$ did exist, then we would get the result by putting $\beta=\alpha_1\phi$ or $\beta=\alpha_2 \psi$, respectively. In other words, it is enough to prove the assertion when $\boldsymbol{\Lambda}=\Lambda$ and either $\Gamma_1=\Gamma$ and $L_1=L$ or $\Gamma_2=\Gamma$ and $L_2=L$.
Let $\boldsymbol{\Sigma}=L_{\boldsymbol{V}}$ with 
\begin{equation*}
    \boldsymbol{V}=\{(x,x^q,\ldots,x^{q^{n-1}}) \colon x \in \F_{q^n}\}
\end{equation*}
and $V$ be an $n$-dimensional $\F_q$-subspace of $\F_{q^n}^n$ such that $\Sigma=L_V$. Since $\boldsymbol{\Sigma}$ and $\Sigma$ are canonical subgeometries, we have that $\langle \boldsymbol{V} \rangle_{\F_{q^n}}=\langle V \rangle_{\F_{q^n}}=\F_{q^n}^n$.
 Moreover, let $\boldsymbol{V} = \langle \textbf{v}_1, \textbf{v}_2,\ldots, \textbf{v}_n \rangle_{\F_q}$ and $V = \langle v_1, v_2,\ldots,v_n \rangle_{\F_q} $. Also let $\Gamma_i=\PG(Z_i,q^n)$, $\boldsymbol{\Lambda} = \PG(R,q^n)$.\\
Note that
$\boldsymbol{W}_i =(\boldsymbol{V} +Z_i)\cap R$  and  $W =(V+Z) \cap R$ are two $n$-dimensional $\F_q$-vector spaces in $R$. 
As either  $\Gamma_1=\Gamma$ and $L_1=\mathrm{p}_{\Gamma_1,\boldsymbol{\Lambda}}(\boldsymbol{\Sigma})=\mathrm{p}_{\Gamma,\boldsymbol{\Lambda}}(\Sigma)=L$ or $\Gamma_2=\Gamma$ and $L_2=\mathrm{p}_{\Gamma_2,\boldsymbol{\Lambda}}(\boldsymbol{\Sigma})=\mathrm{p}_{\Gamma,\boldsymbol{\Lambda}}(\Sigma)=L$, by Theorem \ref{thm:Gamma_L_class}, it follows that there exists a non-singular $\F_{q^n}$-semilinear map $\gamma: R \rightarrow R$ such that  either $\boldsymbol{W}^\gamma_1=W$ or $\boldsymbol{W}_2^\gamma=(\boldsymbol{W}_1^{\perp_{\beta'}})^\gamma=W$ where $\perp_{\beta'}$ is the orthogonal complement map defined by a sesquilinear form $\beta'$ on the lattices of the $\F_{q}$-subspaces of $R \cong \F_{q^n}^2$ seen as a $2n$-dimensional space over $\F_q$. Moreover, $\boldsymbol{W}_1$ and $\boldsymbol{W}_2$ are not on the same orbit under the action of the group of semilinear maps of $R$ in itself, otherwise $L_f$, and hence also $L$, would not have $\GaL$-class two.\\
Assume first that $\boldsymbol{W}_1^\gamma=W$ and let $\Gamma \in \{\Gamma_1, \Gamma_2\}$ and, hence, $Z \in \{Z_1,Z_2\}$.  We show that the case $\boldsymbol{W}_1^\gamma=W$  and $\Gamma_2=\Gamma$  cannot occur. Indeed, if $\boldsymbol{W}_1^\gamma=W$ and $Z=Z_2$,  then the map $\gamma$ can be extended to non-singular semilinear map, saying $\gamma$ again, from $\F^n_{q^n}$ in itself such that $Z_2^\gamma=Z_2$. Then for each vector $v \in \F_{q^n}^n$ and $\mu \in \F_{q^n}$, we have $(\mu v )^\gamma=\mu^{p^\ell}v^\gamma$ for some integer $\ell$. Since $R+Z$ and $\boldsymbol{W}_2+Z$ are direct sums, for any $\textbf{v} \in \boldsymbol{V} \setminus \{\textbf{0}\}$ and $z \in Z$, the intersection $(\langle \textbf{v} \rangle_{\F_{q}} +Z) \cap R$ is a one-dimensional $\F_q$-subspace, say $ \langle \textbf{w} \rangle_{\F_{q}}$ with $\textbf{w} \in \boldsymbol{W}_2$. This implies $h \textbf{w}=\textbf{v}+z'$ for some $z' \in Z$ and $h \in \F_q$ and hence $\textbf{v}+z=h\textbf{w}+(z-z') \in \boldsymbol{W}_2+Z$. Then, $\boldsymbol{W}_2+Z=\boldsymbol{V}+Z$and hence $\boldsymbol{W}_1 \subseteq \boldsymbol{V}+Z_2=\boldsymbol{W}_2+Z_2$. Since $\boldsymbol{W}_1 \subseteq R$ we get that $\boldsymbol{W}_1 \subset 
(\boldsymbol{V}+Z_2) \cap R=\boldsymbol{W}_2$ and hence $\boldsymbol{W}_1=\boldsymbol{W}_2$, a contradiction.\\
Hence, let us suppose that $\boldsymbol{W}_1^\gamma=W$  and $Z=Z_1$. Similar argument shows that $\boldsymbol{W}_1+Z=\boldsymbol{V}+Z$ and $W+Z=V+Z$. 
Then
\begin{equation}
    V+Z=W+Z=\boldsymbol{W}^\gamma_1+Z^\gamma=(\boldsymbol{W}_1+Z)^\gamma=(\boldsymbol{V}+Z)^\gamma
\end{equation}
and hence for any $i=1,\ldots,n$, $\textbf{v}^\gamma_i=v'_i+z_i$ for some $v'_i \in V$ and $z_i \in Z$.
We will show that $v_1',\ldots,v_n'$ is an $\F_q$ basis of $V$. Indeed, suppose that $\sum_{i=1}^n \lambda_i v'_i=0$ for $\lambda_i \in \F_q$. Therefore, $\sum_{i=1}^n \lambda_i \textbf{v}_i^\gamma=\sum_{i=0}^n \lambda_i z_i$. As on the left-hand since $\textbf{v}^\gamma_1,\textbf{v}_2^\gamma,\ldots, \textbf{v}_n^\gamma$ are $\F_q$ independent, it follows that either $\lambda_i=0$ or, saying $z=\sum_{i=1}^n \lambda_i \textbf{v}_i^\gamma=\sum_{i=0}^n \lambda_i z_i$, then $z \in \boldsymbol{V}^\gamma \cap Z$. Hence in the latter case $z^{\gamma^{-1}} \in \boldsymbol{V} \cap Z$, a contradiction. Since $\Sigma=L_V$ is a subgeometry, $v'_1,v'_2,\ldots,v'_n$ are linearly independent over $\F_{q^n}$.
Let $\rho$ be the $\F_{q^n}$-semilinear
map of $\F_{q^n}^n$ such that $\textbf{v}_i^\rho=  v'_i$ for each $i = 1, 2,\ldots, n$ and $(\mu\textbf{v})^\rho=\mu^{p^\ell}\textbf{v}^\rho$    for any $\mu \in \F_{q^n}$. If $P = \langle \textbf{z} \rangle_{\F_{q^n}} \in \Gamma$, then we have $\textbf{z} =\sum^n_{i=1} a_i\textbf{v}_i$ for some $a_i \in \F_{q^n}$. Then 
\begin{equation*}
\begin{split}
\textbf{z}^\rho=\sum_{i=1}^na_i^{p^{\ell}}v'_i=
\sum_{i=1}^na_i^{p^{\ell}}(\textbf{v}_i^\gamma-z_i)=\textbf{z}^\gamma-\sum_{i=1}^na_i^{p^{\ell}}z_i \in Z
\end{split}
\end{equation*}
and hence the collineation induced by $\rho$ maps $\Gamma$ in $\Gamma$ and $\boldsymbol{\Sigma}$ into $\Sigma$.
Finally, a similar argument applies to the case $\boldsymbol{W}_2^\gamma=W$.\\
$(ii) \Rightarrow (i)$. It follows \cite[Theorem 2]{Csajbok_Zanella} and \cite[Theorem 3]{LavrauwVandeVoorde}. 
\end{proof} 

\label{conjgroup}
    Let  $\Sigma_1$ and $\Sigma_2$ be two canonical subgeometries of $\Sigma^*$ and let us consider $\mathbb{G}_1=\langle \sigma_1\rangle$ and $\mathbb{G}_2$ the subgroups of $\mathrm{P}\Gamma\mathrm{L}(n,q^n)$ fixing pointwise the canonical subgeometries $\Sigma_1$ and $\Sigma_2$ of $\Sigma^*$, respectively. If $\beta$ is a collineation of $\Sigma^*$ such that $\Sigma_1^{\beta}=\Sigma_2$, then $\mathbb{G}_2=\mathbb{G}_1^\beta$, i.e. they are conjugate under $\beta$. Indeed, for any point $P \in \Sigma_2$ we have
    \begin{equation*}
    P^{\beta^{-1}\sigma_1 \beta}=((P^{\beta^{-1}})^{\sigma_1 })^\beta=P^{\beta^{-1} \beta}=P
    \end{equation*}
    Analogously, it is straightforward to see that $\mathrm{Aut}(\Sigma_2)= \mathrm{Aut}(\Sigma_1)^\beta$.
Hence, as byproduct of Theorem \ref{thm:embed} and Theorem \ref{thm:move-gamma}, we can state the following.

\begin{proposition}\label{prop:pointgamma}
    Let  $\Sigma \cong \PG(n - 1, q)$ be a canonical subgeometry of $\Sigma^*= \PG(n - 1, q^n)$. Let $L$  be a linear set contained in a line  $\Lambda$ of $\Sigma^*$ such that $L=\mathrm{p}_{\Gamma,\Lambda}(\Sigma)$, where $\Gamma$ is an $(n - 3)$-subspace of $\Sigma^*$ such that
$\Gamma \cap \Sigma = \emptyset= \Gamma \cap \Lambda$. If
 $ L$ is equivalent to a maximum scattered linear set of $\PG(1,q^n)$ not of pseudoregulus type $L_f$, with $f(x)=\sum_{i \in I}a_ix^{q^{si}}$, $m=\deg_{q^s}f(x)$ and $I=\mathrm{supp}_{q^s}f(x)$ such that either $L_f$ is simple or $\Lambda_f=\Lambda_{\hat{f}}$ (cf. \eqref{gammaf} and \eqref{gammahatf}),  then 
there exists a collineation $\sigma$ of $\Sigma^*$ fixing $\Sigma$ pointwise and an imaginary point $P$ of $\Sigma^*$ (wrt $\Sigma)$ such that either
\begin{equation}\label{gammaA}
\Gamma= \langle P^{\sigma^j}, Q_i \colon j \not \in I , j \neq 0 \, \textnormal{and}\, i \in I\setminus\{m\}\rangle, 
\end{equation}
where $Q_i \in \langle  P^{\sigma^i}, P^{\sigma^m} \rangle $, $i \in  I \setminus \{m\}$, or
\begin{equation}\label{gammahat}
\Gamma=\langle P^{\sigma^j}, Q_i \colon j \not \in \hat{I} , j \neq 0 \, \textnormal{and}\, i \in \hat{I}\setminus\{\hat{m}\}\rangle,
\end{equation}
where $Q_i \in \langle  P^{\sigma^i}, P^{\sigma^{\hat{m}}} \rangle $, $i \in  \hat{I} \setminus \{\hat{m}\}$ and $\hat{I}=\mathrm{supp}\hat{f}(x)$, $\hat{m}=\deg_{q^s} \hat{f}(x)$.

\end{proposition}

\begin{proof}
By Theorem \ref{thm:move-gamma}, there exists a collineation $\beta$ of $\Sigma^*$ such that $\mathbf{\Sigma}^{\beta}=\Sigma$ and either $\Gamma = \Gamma^\beta_f$ or  $\Gamma = \Gamma^\beta_{\hat{f}}$ where $\Gamma_f$ and $\Gamma_{\hat{f}}$ are defined as in \eqref{gammaf} or \eqref{gammahatf}. If $\Gamma = \Gamma_f^\beta$, the result follows putting $\sigma=\beta^{-1}\boldsymbol{\sigma}^{s}\beta$ (cf. \eqref{eq:order_n_collineation}), $P=\boldsymbol{P}^\beta$ and $Q_i=\boldsymbol{Q}^\beta_i$ where $\boldsymbol{P}$ and $\boldsymbol{Q}_i$ are the points as in \eqref{gammaf} of the statement of Theorem \ref{thm:embed}, and noting that $P^{\sigma^j}=\boldsymbol{P}^{\beta\sigma^j}=\boldsymbol{P}^{\boldsymbol{\sigma}^{sj}\beta}$, $j=1,\ldots,n-1$. If otherwise $\Gamma = \Gamma_{\hat f}^\beta$ the result instead follows by putting $\sigma=\beta^{-1}\boldsymbol{\sigma}^{s(n-1)}\beta$, $P={\boldsymbol P}^{\beta}$, ${\boldsymbol T}^{\beta}_i=Q_i$ and taking into account Remark \ref{aggiunto}.
\end{proof}

In particular, if $L \subset \Lambda$ is simple,  then $\PGaL(n,q^n)$ acts transitively  on the vertices of any linear set equivalent to $L$, while the group $\Aut(\Sigma)$ acts transitively on the subspaces $\Gamma$ such that $L=\mathrm{p}_{\Gamma,\Lambda}(\Sigma)$.

Finally, we will state Theorem \ref{GPT} in the case of a linear set contained in a projective line of $\PG(n-1,q^n)$.

\begin{theorem} \label{inverseline}
     Let $I\subseteq\{1,\dots, n-1\}$ and $m$ denotes the maximum integer in $I$. Let $\Sigma$ be a canonical subgeometry of $\Sigma^*=\PG(n-1,q^n)$ and consider an $(n-3)$-dimensional subspace $\Gamma$ and a line $\Lambda$ of $\Sigma^*$ such that  $\Gamma \cap \Sigma = \emptyset = \Lambda  \cap \Gamma $. \\
     Assume  $L= \mathrm{p}_{\Gamma,\Lambda}(\Sigma)$ is a linear set  $(1,n-2)$-evasive of rank $n$.
     If there exists a point $P$ with the property that
     \begin{equation*}
         \Gamma=\langle P^{\sigma^j}, Q_i ~|~ j\notin I,\, j \neq 0 \, \text{ and } i\in I\setminus\{m\}\rangle
     \end{equation*}
     where $Q_i \in \langle P^{\sigma^m}, P^{\sigma^i}\rangle$ and $\sigma$ is a generator of the subgroup of $ \PGaL(n,q^n)$ fixing $\Sigma$ pointwise,
     then $P$ is an imaginary point (wrt $\Sigma)$ and  $L \cong L_f$ where $f(x)=\sum_{i \in I} a_i x^{q^{si}}$ and $\gcd(s,n)=1$ for some integer $1 \leq s \leq n-1$.
\end{theorem}

\begin{remark}
\begin{enumerate}
\item []
\item [$(i)$]  Note that if $L_f$ is equivalent to $L=\mathrm{p}_{\Gamma,\Lambda}(\Sigma)$ with $\Gamma$ and $\Lambda$ as Proposition \ref{directline},  $f$ is a permutation polynomial if and only if $\Sigma \cap \langle \Gamma, P \rangle=\emptyset$. 

 \item  [$(ii)$]  The imaginary point $P$ (wrt $\Sigma$) in Theorem  \ref{inverseline} is not unique.  For instance, let  $L^{4,6}_{\delta}$ be  as in \eqref{tri11}, $\delta^2+\delta=1$, then  $\delta \in \F_{q^2}$, and let us consider the point $A$ with homogeneous coordinates $(0,-1,0,-1,0,1+\delta)$ in the frame  $(P,P^\sigma,\ldots,P^{\sigma^{5}},U)$ with $P=\langle v \rangle_{\F_{q^6}}$ and $U=\langle v+v^\sigma+\ldots+v^{\sigma^{n-1}} \rangle_{\F_{q^6}}$. Then,
    \begin{equation*}
    \begin{vmatrix}
    0 & -1 & 0 & -1 & 0 & 1+\delta\\
    1+\delta^q & 0 & -1 &  0 & -1 & 0\\ 
    0 & 1+\delta & 0 & -1 &  0 & -1\\ 
    -1 & 0 & 1+\delta^q & 0 & -1  &0\\
    0 &-1& 0& 1+\delta& 0 &-1\\
    -1 & 0 & -1 & 0 & 1+\delta^q & 0 
    \end{vmatrix}
    =(1-\delta) (2+\delta)^2 ( \delta^q-1) (2 + \delta^q)^2 \neq 0.
    \end{equation*}
and hence $\dim \mathcal{L}_A=5$ and the solid $\Gamma$ as in \eqref{gamma} is spanned by the points $A^{\sigma^2},A^{\sigma^4}$, $Q'_1=\langle(0,0,-2-\delta^q,0,\delta^q+1,0)\rangle_{\F_{q^6}}$ and $Q'_2=\langle(0,0,2+\delta^q,0,-2-\delta^q,0)\rangle_{\F_{q^6}}$, where $Q_1' \in \langle A^{\sigma}, A^{\sigma^5} \rangle$ and $Q_2' \in \langle A^{\sigma^3}, A^{\sigma^5}\rangle$.
\end{enumerate}
 \end{remark}

\section{Maximum scattered linear sets which are neither of pseudoregulus nor of LP type}

As we have seen in the previous section if $L = \mathrm{p}_{\Gamma,\Lambda}(\Sigma)$ for a given subgeometry $\Sigma$ and suitable subspaces $\Gamma$ and $\Lambda$ of $\PG(n-1,q^n)$, then, in general, the intersection number of $\Gamma$ does not characterise $L$ by its own. 

In this section, we will deepen the study of projection vertices of linear sets of $PG(1,q^n)$ known to date, which are neither of pseudoregulus nor of LP type, in order to provide characterization results for all of them in term of certain properties of relevant vertices. 

In this regard, we will use a notion from classical projective geometry: the cross-ratio of four points of the projective line.\\
Let $A,B,C,D$ be four points of $\PG(1,q^n)$ such that $A, B, C$ are distinct; the \textit{cross-ratio} $(A;B;C;D)$ of these points is defined as
\begin{equation*}
    (A;B;C;D)=\frac{\begin{vmatrix}
     a_0 &  c_0  \\
      a_1 &  c_1 
 
    \end{vmatrix}
    \begin{vmatrix}
        b_0 & d_0 \\
        b_1 & d_1
    \end{vmatrix}}
    {    \begin{vmatrix}
    a_0 & d_0\\
a_1 & d_1
    \end{vmatrix}\begin{vmatrix}
      b_0 & c_0\\
b_1 & c_1  
    \end{vmatrix}},
\end{equation*}
where $A=\langle (a_0,a_1)\rangle_{\F_{q^n}},B=  \langle (b_0,b_1)\rangle_{\F_{q^n}}, C= \langle (c_0,c_1)\rangle_{\F_{q^n}},D= \langle (d_0,d_1)\rangle_{\F_{q^n}}$ , respectively.
Setting $\infty:=\frac{1}{0}$, we have
\begin{equation*}
(A;B;C;A) = \infty, \,(A;B;C;B)=0,\, (A;B;C;C)=1
\end{equation*}
Moreover, by the definition
we have

\begin{equation*}
 (A;B;C;D)=(B;A;D;C)=(C;D;A;B)=(D;C;B;A); 
\end{equation*}
and if $(A;B;C;D) = \kappa$, all the
possible values of the cross-ratio are
\begin{equation*}
\begin{matrix}
\noindent (A;B;D;C)=1/\kappa,&(A;C;B;D)=1-\kappa,
&(A;D;B;C)=(\kappa-1)/\kappa, \vspace{0.2cm} \\
(A;D;C;B) = \kappa/(\kappa-1), &(A;C;D;B)=1/(1-\kappa).
\end{matrix}
\end{equation*}
We will say that the points pair $\{A,B\}$  \textit{harmonically separates} the points pair $\{C,D\}$ if the cross-ratio $(A;B;C;D)=-1$, see \cite[Chapter 4]{Casse} and \cite{hirsh}.
Finally, it is straightforward to see that if $\phi \in \PGaL(2,q^n)$ with companion automorphism $\tau$, then 
\begin{equation*}
(A^\phi;B^\phi;C^\phi;D^\phi) = (A;B;C;D)^\tau.
\end{equation*}

\subsection{A characterization of $L_{\ell,\eta}^{3,n}$, $n \in \{6,8\}$ }

Let 
$L_{\ell,\eta}^{3,n}$ be the linear set defined as in \eqref{ex:CMPZ} 
and denote by
$U_{\ell,\eta}^{3,n} \subseteq \F_{q^n} \times \F_{q^n}$ its underlying vector space.
According to \cite{Csajbok-Marino-Polverino-Zanella}, $U^{3,n}_{\ell, \eta}$ is $\GL(2, q^n)$-equivalent to $U^{3,n}_{n-\ell, \eta^{q^{n-s}}}$
and to $U^{3,n}_{\ell+n/2,\eta^{-1}}$ as well. Thus, as indicated in \cite[Section 4]{Csajbok_Marino_Zullo}, we may assume $\ell < n/4$, and $\gcd(\ell, n/2) = 1$ and hence reduce the study to linear sets $L^{3,n}_{\ell,\eta}$ with $\ell = 1$.
Moreover, in \cite[Proposition 4.1 and 4.2]{Csajbok_Marino_Zullo}, it is proven that in both cases, the relevant linear set is simple.

\begin{theorem}\label{binom6}
  Let $L$ be a $(1,4)_q$-evasive linear set of rank $6$ contained in a line $\Lambda$ of $\PG(5, q^6)$ and  let $\Sigma$ be a canonical subgeometry of $\PG(5,q^6)$. Then $L$ is equivalent to a maximum scattered linear set of type $L_{1,\eta}^{3,6}$, for some $\eta \in \F_{q^6}$, if and only if for each solid $\Gamma$ of $\PG(5, q^6)$ such that $L = \mathrm{p}_{\Gamma ,\Lambda}(\Sigma)$, the following holds:
\begin{enumerate}
\item  there exist a generator $\sigma$ of the subgroup of $\PGaL(6, q^6)$ fixing $\Sigma$ pointwise and a  point $P \in \PG(5,q^6)$ such that
\begin{equation*}
    \Gamma = \langle P^{\sigma^i},Q : i \in \{2,3,5\} \rangle
\end{equation*} 
with $Q \in \langle P^{\sigma},P^{\sigma^4} \rangle$, 
\item the points $P^\sigma,P^{\sigma^4},Q, Q^{\sigma^3}$ are collinear, and the equation
\begin{equation}\label{equation}
    Y^2-(\Tr_{q^3/q}(\gamma)-1)Y+\N_{q^3/q}(\gamma)=0, 
\end{equation}
where $\gamma=(P^{\sigma};Q;Q^{\sigma^3};P^{\sigma^4})^{\tau}$ for some $\tau \in \mathrm{Aut}(\F_{q^6})$, admits two solutions in $\F_q$.
\end{enumerate}
\end{theorem}
\begin{proof}
Let us suppose that $L$ is equivalent to a maximum scattered linear set of type  $L_{1,\eta}^{3,6}$. By Proposition \ref{prop:pointgamma} and since 
$L$ is simple, Point {\em 1} holds true.
By Theorem \ref{thm:embed}, let $\Gamma_f=\langle \boldsymbol{P}^{\boldsymbol{\sigma}^{i}}, \boldsymbol{Q}: i \in \{2,3,5\} \rangle$, where $f(x)= \eta x^q+ x^{q^4}$, $\boldsymbol{P}$ an imaginary point of $\PG(5,q^6)$ (wrt $\boldsymbol{\Sigma}$), $\boldsymbol{Q} \in \langle \boldsymbol{P}^{\boldsymbol{\sigma}}, \boldsymbol{P}^{\boldsymbol{\sigma}^4} \rangle $, $P=\boldsymbol{P}^{{\beta}}$ and $Q=\boldsymbol{Q}^{\beta}$ where $\beta$ is the collineation of $\PG(5,q^6)$ such that $\Gamma_f^\beta=\Gamma$ and $\boldsymbol{\Sigma}^\beta=\Sigma$. \\
Clearly since $L_{1,\eta}^{3,6}$ is simple, all possible projection vertices for such a linear set are equivalent to $\Gamma_f$ under the action of $\PGaL(5,q^6)$. Since $\beta$ preserves the collinearity of points, in order to show Point \textit{2}, it is enough to prove  that $\boldsymbol{P}^{\boldsymbol{\sigma}},\boldsymbol{P}^{\boldsymbol{\sigma}^4},\boldsymbol{Q},\boldsymbol{Q}^{\boldsymbol{\sigma}^3}$ are collinear. Indeed, note that $\boldsymbol{\sigma}^3$ fixes the line $\langle \boldsymbol{P}^{\boldsymbol{\sigma}}, \boldsymbol{P}^{\boldsymbol{\sigma}^4} \rangle$, therefore $\boldsymbol{Q}^{\boldsymbol{\sigma}^{3}} \in \langle \boldsymbol{P}^{\boldsymbol{\sigma}}, \boldsymbol{P}^{\boldsymbol{\sigma}^4} \rangle$. Let us suppose that $\boldsymbol{P}=\langle \textbf{v} \rangle_{\F_{q^6}}$ with $\textbf{v}$ a non-zero vector belonging to $\F_{q^6}^6$. If $\boldsymbol{Q}=\langle \textbf{v}^{\boldsymbol{\sigma}}-\eta \textbf{v}^{\boldsymbol{\sigma}^4} \rangle_{\F_{q^6}}$, then $\boldsymbol{Q}^{\boldsymbol{\sigma}^3}=\langle \textbf{v}^{\boldsymbol{\sigma}^4}-\eta^{q^3} \textbf{v}^{\boldsymbol{\sigma}}  \rangle_{\F_{q^6}}$. 
Hence,
  $$(\boldsymbol{P}^{\boldsymbol{\sigma}};\boldsymbol{P}^{\boldsymbol{\sigma}^4};\boldsymbol{Q};\boldsymbol{Q}^{\boldsymbol{\sigma}^3})=
  \frac{\begin{vmatrix} 1 & 1 \\ 0 & -\eta \end{vmatrix} \begin{vmatrix} 0 & -\eta^{q^3} \\ 1 & 1 \end{vmatrix}}{\begin{vmatrix} 1 & -\eta^{q^3} \\ 0 & 1 \end{vmatrix} \begin{vmatrix} 0 & 1 \\ 1 & -\eta \end{vmatrix}}=\eta^{q^3+1}.$$
Now, the linear set $L \cong L_{1,\eta}^{3,6}$ is equivalent to $L^{3,6}_{2,\eta^{-q^5}}$ which is scattered as well. Then by \cite[Theorem 7.3]{polzullo}, the equation
\begin{equation*}
     Y^2-(\Tr_{q^3/q}(\alpha)-1)Y+\N_{q^3/q}(\alpha)=0 
 \end{equation*}
admits two solutions in $\F_q$, where $\alpha=-\frac{(\eta^{-q^5})^{q^3+1}}{1-(\eta^{-q^5})^{q^3+1}}=\frac{1}{1-\eta^{q^2+q^5}}$. Moreover, we have $\alpha^q=(\boldsymbol{P}^{\boldsymbol{\sigma}};\boldsymbol{Q};\boldsymbol{Q}^{\boldsymbol{\sigma}^3};\boldsymbol{P}^{\boldsymbol{\sigma}^4})$. 
Also, if $\rho$ is the companion automorphism of $\beta$, then
\begin{equation*}\label{crossratio}
(P^{\sigma};Q;Q^{\sigma^3};P^{\sigma^4})=(\boldsymbol{P}^{\beta \sigma};\boldsymbol{Q}^\beta; \boldsymbol{Q}^{\beta {\sigma}^3};\boldsymbol{P}^{\beta \sigma^4})=(\boldsymbol{P}^{\boldsymbol{\sigma} \beta};\boldsymbol{Q}^{\beta}; \boldsymbol{Q}^{\boldsymbol{\sigma}^3 \beta };\boldsymbol{P}^{\boldsymbol{\sigma}^4 \beta })=(\boldsymbol{P}^{\boldsymbol{\sigma}};\boldsymbol{Q}; \boldsymbol{Q}^{\boldsymbol{\sigma}^3 };\boldsymbol{P}^{\boldsymbol{\sigma}^4})^\rho.
\end{equation*}
Finally, since $\mathrm{Tr}_{q^3/q}(\alpha)=\mathrm{Tr}_{q^3/q}(\alpha^q)$ and $\mathrm{N}_{q^3/q}(\alpha)=\mathrm{N}_{q^3/q}(\alpha^q),$ Point $\textit{2}$ follows with $\gamma=\alpha^q$ and $\tau=\rho^{-1}$.

Suppose now that Points \textit{1} and $\textit{2}$ hold true. Theorem \ref{inverseline} implies that $P=\langle v \rangle_{\F_{q^6}}$ is an imaginary point (wrt $\Sigma$), $Q=\langle v^\sigma - \eta v^{\sigma^4}\rangle_{\F_{q^6}}$ and $L \cong L_f$, $f(x)=\eta x^{q^{s}}+x^{q^{4s}}$ for some $\eta \in \F_{q^6}^*$ and $s \in \{1,5\}$. Since $P,P^{\sigma},Q,Q^{\sigma^3}$ are distinct, $\mathrm{N}_{q^6/q^3}(\eta) \neq 1$. If $s=1$, since $L_f=L_{\hat{f}}$ is equivalent to $L_{g_1}$ with $g_1(x)=\eta^{-q^5}x^{q^2} + x^{q^5}$.  Point $\textit{2}$ implies that the equation
 \begin{equation*}
     Y^2-(\Tr_{q^3/q}(\alpha_1)-1)Y+\N_{q^3/q}(\alpha_1)=0 
 \end{equation*}
admits two solutions in $\F_q$, where $\alpha_1=-\frac{(\eta^{-q^5})^{q^3+1}}{1-(\eta^{-q^5})^{q^3+1}}$.
Hence, by \cite[Theorem 7.3]{polzullo}, $L_{g_1}$ is scattered and $L$ is equivalent to  a scattered linear set of type $L_{1,\eta}^{3,6}$. If $s=5$ and hence $f(x)=\eta x^{q^5}+x^{q^2}$, $L_f \cong L_{g_2}$ with $g_2(x)=x^{q^5}+\eta^{-1} x^{q^2}$. Point $\textit{2}$ implies that the equation
 \begin{equation*}
     Y^2-(\Tr_{q^3/q}(\alpha_2)-1)Y+\N_{q^3/q}(\alpha_2)=0 
 \end{equation*}
admits two roots in $\F_q$ with $\alpha_2=-\frac{\eta^{-(q^3+1)}}{1-\eta^{-(q^3+1)}}$. Then, by \cite[Theorem 7.3]{polzullo}, $g_2(x)$ is scattered and hence $L$ is scattered as well, and it is of the type $L^{3,6}_{2,\eta^{-1}} \cong L_{1,\eta^q}^{3,6}$.
\end{proof}

In the same way we get

\begin{theorem}
     Let $L$ be an $(1,6)_q$-evasive linear set of rank $8$ contained in a line $\Lambda$ of $\PG(7, q^8)$, with $q\leq 11$ or $q\geq1039891$ odd. Then, $L$ is equivalent to a maximum scattered linear set of type $L_{r,\eta}^{3,8}$, for some $\eta \in \F_{q^8}$ and $\gcd(s,8)=1$, if and only if for each $5$-dimensional subspace $\Gamma$ of $\PG(7, q^8)$ such that $L = \mathrm{p}_{\Gamma ,\Lambda}(\Sigma)$, the following hold:
     \begin{enumerate}
         \item  there exists a generator $\sigma$ of the subgroup of $\PGaL(8, q^8)$ fixing $\Sigma$ pointwise and a point $P \in \PG(7,q^8)$ such that
\begin{equation*}
    \Gamma = \langle P^{\sigma^i},Q : i  \in \{2,3,4,6,7\} \rangle
\end{equation*} 
with $Q \in \langle P^{\sigma},P^{\sigma^5} \rangle$, 
\item the points $P^\sigma,P^{\sigma^5},Q,Q^{\sigma^4}$ are collinear and pair $\{Q, Q^{\sigma^4}\}$ harmonically separates the pair $\{P^{\sigma},P^{\sigma^{5}}\}$ .
     \end{enumerate}
\end{theorem}
\begin{proof}
    Let $L$ be equivalent to a maximum scattered linear set of type $L_{1,\eta}^{3,8}$. Since $L$ is simple, by Proposition \ref{prop:pointgamma}, Point {\em 1} holds true. By Theorem \ref{thm:embed}, let $\Gamma_f=\langle \boldsymbol{P}^{\boldsymbol{\sigma}^{i}}, \boldsymbol{Q}: i \in \{2,3,4,6,7\} \rangle$, $f(x)=\eta x^q+ x^{q^5}$, with $\boldsymbol{P}$ an imaginary point of $\PG(7,q^8)$ (wrt $\boldsymbol{\Sigma}$), $\boldsymbol{Q} \in \langle \boldsymbol{P}^{\boldsymbol{\sigma}},\boldsymbol{P}^{\boldsymbol{\sigma}^5} \rangle$, $P=\boldsymbol{P}^{{\beta}}$ and $Q=\boldsymbol{Q}^{\beta}$ where $\beta$ is the collineation of $\PG(7,q^8)$ such that $\boldsymbol{\Sigma}^\beta=\Sigma$ and $\Gamma_f^\beta=\Gamma$. \\As before, all possible projection vertices for such a linear set are equivalent to $\Gamma_f$ under the action of $\PGaL(7,q^8)$. Since $\beta$ preserves collinearity of points and the property for a pair of points to harmonically separate another pair, in order to show Point \textit{2}, it is enough to prove that $\boldsymbol{P}^{\boldsymbol{\sigma}},\boldsymbol{P}^{\boldsymbol{\sigma}^5},\boldsymbol{Q},\boldsymbol{Q}^{\boldsymbol{\sigma}^4}$ are collinear and $(\boldsymbol{P}^{\boldsymbol{\sigma}};\boldsymbol{P}^{\boldsymbol{\sigma}^5};\boldsymbol{Q};\boldsymbol{Q}^{\boldsymbol{\sigma}^4})=-1$. 
   Clearly, $\boldsymbol{\sigma}^4$ fixes the line $\langle \boldsymbol{P}^{\boldsymbol{\sigma}}, \boldsymbol{P}^{\boldsymbol{\sigma}^5} \rangle$ and $\boldsymbol{Q}^{\boldsymbol{\sigma}^{4}} \in \langle \boldsymbol{P}^{\boldsymbol{\sigma}}, \boldsymbol{P}^{\boldsymbol{\sigma}^5} \rangle$. Now, let us suppose that $\boldsymbol{P}=\langle \textbf{v} \rangle_{\F_{q^6}}$ with $\textbf{v}$ a non-zero vector belonging to $\F_{q^8}^8$. If $\boldsymbol{Q}=\langle \textbf{v}^{\boldsymbol{\sigma}}-\eta \textbf{v}^{\boldsymbol{\sigma}^5} \rangle_{\F_{q^8}}$, then $\boldsymbol{Q}^{\boldsymbol{\sigma}^4}=\langle \textbf{v}^{\boldsymbol{\sigma}^5}-\eta^{q^4} \textbf{v}^{\boldsymbol{\sigma}}  \rangle_{\F_{q^8}}$ with $\eta \in \F_{q^8}^*$.  Since we have assumed that either $q\leq 11$ or $q\geq1039891$, by \cite[Theorem 1.1]{Timpanella_Zini} we have $\mathrm{N}_{q^8/q^4}(\eta)=-1$.
Then
$$(\boldsymbol{P}^{\boldsymbol{\sigma}};\boldsymbol{P}^{\boldsymbol{\sigma}^5};\boldsymbol{Q};\boldsymbol{Q}^{\boldsymbol{\sigma}^4})=
  \frac{\begin{vmatrix} 1 & 1 \\ 0 & -\eta \end{vmatrix} \begin{vmatrix} 0 & -\eta^{q^4} \\ 1 & 1 \end{vmatrix}}{\begin{vmatrix} 1 & -\eta^{q^4} \\ 0 & 1 \end{vmatrix} \begin{vmatrix} 0 & 1 \\ 1 & -\eta \end{vmatrix}}=\eta^{q^4+1}=-1.$$
and hence Point \textit{2}.\\
Now suppose  Point \textit{1} and \textit{2} hold true. By Theorem \ref{inverseline}, since $L$ is $(1,6)_q$-evasive, $L \cong L_f$ with $f(x)=\eta x^{q^s}+x^{q^{5s}}$, $s \in \{1,3,5,7\}$ and $\eta \in \F_{q^8}^*$. Moreover, the point $P= \langle v \rangle_{\F_{q^8}}$ is an imaginary point (wrt $\Sigma$). Since $Q= \langle v^\sigma- \eta v^{\sigma^5} \rangle_{\F_{q^8}}$ and  $-1=(P^{\sigma};P^{\sigma^{5}};Q;Q^{\sigma^4})=\mathrm{N}_{q^8/q^4}(\eta)$, again by \cite[Theorem 1.1]{Timpanella_Zini}, $L_f$, and hence $L$, is equivalent to a scattered linear set of type $L_{r,\eta}^{3,8}$.
\end{proof}

\subsection{A characterization of $L_{\delta}^{4,6}$}
Let $L_{\delta}^{4,6}$ be the linear set defined as in \eqref{tri11} for $q$ odd  and $\delta^2+\delta=1$. In \cite{Csajbok_Marino_Zullo}, it is showed that for $q \equiv 0, \pm 1 \pmod 5$, this is scattered. Remaining
cases for $q$ odd were treated later in \cite{Marino_Montanucci_Zullo}. 

In the following, exploiting same arguments used by Csajb\'{o}k, Marino and Zullo in \cite[Thereom 5.7]{Csajbok_Marino_Zullo}), we firstly determine conditions under which the maximum scattered linear set $L_{\delta}^{4,6}$ is simple. 

\begin{proposition} 
The linear set $L_{\delta}^{4,6}$, with $q=p^e$ odd and $\delta^2+\delta=1$, is simple if and only if either $q \equiv 0 \pmod 5$ or $p \equiv \pm 2 \pmod 5$ for some prime $p$.
\end{proposition}
\begin{proof}
     To this aim it is enough to observe that the missing case in \cite[Thereom 5.7]{Csajbok_Marino_Zullo}) is when $q = p^{2h+1}$ and $p \equiv \pm 2 \pmod 5$. In this case, $\delta \in \F_{p^2} \setminus \F_p$. So, $\delta$ and $\delta^p$ are two distinct roots of $X^2 +X-1=0$ and $X^q +X+1 = 0$  and we get that $U_f$ and  $U_{\hat{f}}$ with $f(x)=x^q+x^{q^3}+\delta x^{q^5}$ are $\GaL(2, q^6)$-equivalent.
\end{proof}

Hence, we can prove the following results.
\begin{theorem}
Let $L$ be a $(1,4)_q$-evasive linear set of rank $6$ contained in a line $\Lambda$ of $\PG(5, q^6)$ and  let $\Sigma$ be a canonical subgeometry of $\PG(5,q^6)$, $q$ odd.
Then $L$ is equivalent to a maximum scattered linear set of type $L_{\delta}^{4,6}$ if and only if for each solid $\Gamma$ of $\PG(5,q^6)$ with $\Gamma \cap \Sigma = \Gamma \cap \Lambda = \emptyset$ such that $L=\mathrm{p}_{\Gamma,\Lambda}(\Sigma)$, the following hold:
    \begin{enumerate}
        \item there exist a generator $\sigma$ of the subgroup of $\PGaL(6,q^6)$ fixing $\Sigma$ pointwise and  a point $P=\langle v \rangle_{\F_{q^6}}$, $v$ a non-zero vector of $\F_{q^6}^6$, such that
            \begin{equation}\label{1solid-tri}
            \Gamma=\langle P^{\sigma^2},P^{\sigma^4},Q,R \rangle\end{equation}
        with $Q \in \langle P^{\sigma},P^{\sigma^5} \rangle$ and $R \in \langle P^{\sigma^3},P^{\sigma^5} \rangle$;
        \item the point $C= \langle v^{\sigma}-v^{\sigma^3} \rangle_{\F_{q^6}}$  belongs to $\Gamma$;
    
      \item the points $P^\sigma,P^{\sigma^{5}}, C^{\sigma^4}$ and $Q$ are collinear and the cross-ratio $(P^{\sigma};P^{\sigma^5};C^{\sigma^4};Q)$ is in $\F_{q^2}$; 
       \item  the points $C,P^{\sigma^5}$ and $\langle Q,Q^{\sigma^2}\rangle\cap \langle R, R^{\sigma^2}\rangle$ are collinear.
    \end{enumerate}    

Moreover, the vertices $\Gamma$ such that $L=\mathrm{p}_{\Gamma,\Lambda}(\Sigma)$ belong to the same orbit under the action of $\Aut(\Sigma)$ if and only if $q \equiv 0 \pmod 5$ or $p \equiv \pm 2 \pmod 5$ for some prime $p$.
\end{theorem}

\begin{proof}
    Let us suppose that $L$ is equivalent to a maximum scattered linear set of type $L_{\delta}^{4,6}$, $\delta^2+\delta=1$.  Since $\Lambda_f= \Lambda_{\hat{f}}$ where $f(x)=x^q+x^{q^3}+\delta x^{q^5}$, by Proposition  \ref{prop:pointgamma}, there exists a collineation $\beta$ such that $\boldsymbol{\Sigma}^\beta=\Sigma$ and either $\Gamma_f^\beta=\Gamma$ or $\Gamma_{\hat{f}}^\beta=\Gamma$, an imaginary point $\boldsymbol{P}=\langle \textbf{v} \rangle_{\F_q^6}$ such that 
    \begin{equation*}
        \Gamma_f=\langle \boldsymbol{P}^{\boldsymbol{\sigma^2}},\boldsymbol{P}^{\boldsymbol{\sigma}^4}, \boldsymbol{Q}_1,\boldsymbol{R}_1 \rangle
    \end{equation*}
where
        $\boldsymbol{Q}_1=\langle \textbf{v}^{\boldsymbol{\sigma}^5}- \delta \textbf{v}^{\boldsymbol{\sigma}} \rangle_{\F_{q^6}}$ and $\boldsymbol{R}_1= \langle \textbf{v}^{\boldsymbol{\sigma}^5}- \delta \textbf{v}^{\boldsymbol{\sigma}^3}  \rangle_{\F_{q^6}}$
and 
    \begin{equation*}
        \Gamma_{\hat{f}}=\langle \boldsymbol{P}^{\boldsymbol{\sigma^2}},\boldsymbol{P}^{\boldsymbol{\sigma}^4}, \boldsymbol{Q}_2,\boldsymbol{R}_2 \rangle
    \end{equation*}
where
        $\boldsymbol{Q}_2=\langle \delta^q\textbf{v}^{\boldsymbol{\sigma}^5}-\textbf{v}^{\boldsymbol{\sigma}} \rangle_{\F_{q^6}}$ and $\boldsymbol{R}_2= \langle \textbf{v}^{\boldsymbol{\sigma}^5}-\textbf{v}^{\boldsymbol{\sigma}^3}  \rangle_{\F_{q^6}}$.
    Let us consider the point  $\boldsymbol{C}_1= \langle \textbf{v}^{\boldsymbol{\sigma}}- \textbf{v}^{\boldsymbol{\sigma}^3} \rangle_{\F_{q^6}}$. Clearly $\boldsymbol{C}_1 \in \Gamma_f$ and it is easy to see that  the point $\boldsymbol{P}^{\boldsymbol{\sigma}},\boldsymbol{P}^{\boldsymbol{\sigma}^5},\boldsymbol{C}_1^{\boldsymbol{\sigma}^4}$ and $\boldsymbol{Q}_1$ are collinear. The cross-ratio $(\boldsymbol{P}^{\boldsymbol{\sigma}};\boldsymbol{P}^{\boldsymbol{\boldsymbol{\sigma}}^5};\boldsymbol{C}_1^{\boldsymbol{\boldsymbol{\sigma}}^4};\boldsymbol{Q}_1)=\delta \in \F_{q^2}$ and since $\delta^2+\delta=1$, elementary argument also shows that point 4 holds true for the points $\boldsymbol{C}_1,\boldsymbol{P}^{\boldsymbol{\sigma}^5}$ and $\langle \boldsymbol{Q}_1,\boldsymbol{Q}^{\boldsymbol{\sigma}^2}_1\rangle\cap \langle \boldsymbol{R}_1, \boldsymbol{R}_1^{\boldsymbol{\sigma}^2}\rangle$ are collinear. Then, if $\Gamma_f^\beta=\Gamma$, putting $\sigma=\boldsymbol{\sigma}^\beta$, $\boldsymbol{P}^\beta=P$, $\boldsymbol{Q}_1^\beta=Q,\boldsymbol{R}_1^\beta=R$ and $\boldsymbol{C}_1^\beta=C$, Point (\textit{1-4}) hold true for $\Gamma$ as well.\\
If $\Gamma^\beta_{\hat{f}}=\Gamma$, let us consider the point $\boldsymbol{S}=\langle \delta^q\textbf{v}^{\boldsymbol{\sigma}^3}-\textbf{v}^{\boldsymbol{\sigma}}\rangle_{\F_{q^6}} \in \Gamma_{\hat{f}} \cap \langle \boldsymbol{P}^{\boldsymbol{\sigma}}, \boldsymbol{P}^{\boldsymbol{\sigma}^3} \rangle_{\F_{q^6}}$. Clearly, the solid $      \Gamma_{\hat{f}}=\langle \boldsymbol{P}^{\boldsymbol{\sigma^2}},\boldsymbol{P}^{\boldsymbol{\sigma}^4}, \boldsymbol{Q}_2,\boldsymbol{S} \rangle$. Then, it is easy to see that  the point $\boldsymbol{P}^{\boldsymbol{\sigma}},\boldsymbol{P}^{\boldsymbol{\sigma}^5},\boldsymbol{R}_2^{\boldsymbol{\sigma}^2}$ and $\boldsymbol{Q}_2$ are collinear and their cross-ratio is equal to $\delta^{-q}$ which belongs to $\F_{q^2}$. Since $\delta^2+\delta=1$, elementary argument shows that the points $\boldsymbol{R}_2,\boldsymbol{P}^{\boldsymbol{\sigma}}$ and $\langle \boldsymbol{Q}_2,\boldsymbol{Q}_2^{\boldsymbol{\sigma}^4}\rangle \cap \langle \boldsymbol{S}, \boldsymbol{S}^{\boldsymbol{\sigma}^4}\rangle$ are collinear. 

Then, if $\Gamma_{\hat{f}}^\beta=\Gamma$, putting $\sigma=\boldsymbol{\sigma}^{5\beta}$, $\boldsymbol{P}^\beta=P$, $\boldsymbol{Q}^\beta_2=Q,\boldsymbol{S}^\beta=R$ and $\boldsymbol{R}_2^\beta=C$, Point (\textit{1-4}) hold for $\Gamma$.\\
Now we suppose that Point $(\textit{1-4})$ hold true then by Theorem \ref{inverseline}, $L \cong L_f$ with $f(x)=x^{q^s}+ \gamma x^{q^ {3s}}+\delta x^{q^{5s}} \in \F_{q^6}[x]$ with $s \in \{1,5\}$ and $P= \langle v \rangle_{\F_{q^6}}$ is an imaginary point. 
 Then $Q=\langle v^{\sigma^5}-\delta  v^{\sigma}\rangle_{\F_{q^6}}$ and $R= \langle \gamma v^{\sigma^5}-\delta v^{\sigma^3}\rangle_{\F_{q^6}}$. 
   Since $C=\langle v^{\sigma}-v^{\sigma^3} \rangle_{\F_{q^6}}\in \Gamma$, then $\gamma=1$. By Point \textit{3}, it is easy to see that the cross-ratio  $(P^{\sigma};P^{\sigma^5};C^{\sigma^4};Q)=\delta$, and hence $\delta \in \F_{q^2}$.  Moreover, by Point $\textit{4}$, the lines $\langle Q, Q^{\sigma^2} \rangle$ and $\langle R, R^{\sigma^2} \rangle$ meet at the point $\langle v^{\sigma} -(\delta^2+\delta)v^{\sigma^3}+v^{\sigma^5} \rangle_{\F_{q^6}}$ and the collinearity of this point, $C$ and $P^{\sigma^5}$ implies
\begin{equation}
0=
\begin{vmatrix}
0 & 0 & 1 \\
1 & -1 & 0 \\
1 & -(\delta^2+\delta) & 1
\end{vmatrix}
=-\delta^2-\delta+1,
\end{equation}
obtaining $L \cong L_{\delta}^{4,6}$.
\end{proof}

The family of linear sets of type $L^{4,6}_\delta$ was investigated for also for $q$ even in \cite{BartLongMarTimp}. The conditions assuring the property of being scattered in this case are rather demanding, and it seems quite difficult to translate them in geometric terms, see \cite[Theorem 2.9]{BartLongMarTimp}. However, by using arguments similar to those used in the theorem above, we can state the following result.

\begin{theorem}
Let $L$ be a $(1,4)_q$-evasive linear set of rank $6$ contained in a line $\Lambda$ of $\PG(5, q^6)$, $q$ even and  let $\Sigma$ be a canonical subgeometry of $\PG(5,q^6)$, $q$ even.
If $L$ is equivalent to a maximum scattered linear set of type $L_{\delta}^{4,6}$, then for each solid $\Gamma$ of $\PG(5,q^6)$ with $\Gamma \cap \Sigma = \Gamma \cap \Lambda = \emptyset$ such that $L=\mathrm{p}_{\Gamma,\Lambda}(\Sigma)$, the following holds:
    \begin{enumerate}
        \item there exist a generator $\sigma$ of the subgroup of $\PGaL(6,q^6)$ fixing $\Sigma$ pointwise and  a point $P=\langle v \rangle_{\F_{q^6}}$, $v$ a nonzero vector of $\F_{q^6}^6$, such that the solid
            \begin{equation}\label{solid-tri}
            \Gamma=\langle P^{\sigma^2},P^{\sigma^4},Q,R \rangle\end{equation}
        with $Q \in \langle P^{\sigma},P^{\sigma^5} \rangle$ and $R \in \langle P^{\sigma^3},P^{\sigma^5} \rangle$;
        \item the point $C= \langle v^{\sigma}-v^{\sigma^3} \rangle_{\F_{q^6}}$ belongs to $\Gamma$;
    
      \item the point $P^\sigma,P^{\sigma^{5}}, C^{\sigma^4}$ and $Q$ are collinear and the cross-ratio $(P^{\sigma};P^{\sigma^5};C^{\sigma^4};Q)$ does not belong to any subfield of $\F_{q{^6}}$.
    \end{enumerate}    
\end{theorem}

\subsection{A characterization of $L_{s,m,h}^{5,t}$}

This section is devoted to characterize via geometrical properties of their projection vertex the scattered linear set of type $L_{s,m,h}^{5,t}$ as in \eqref{quadrinomial} with 

\begin{equation*}
\psi_{t,s,m,h}(x)=x^{q^{s(t-1)}}+h^{1-q^{s(2t-1)}}x^{q^{s(2t-1)}}+m \left (x^{q^s}-h^{1-q^{s(t+1)}}x^{q^{s(t+1)}}\right ),
\end{equation*}
where $t\geq3$ and $(m,h)  \in \F_{q^{t}}^* \times \F_{q^{2t}}^*$. We will proceed by splitting our discussion in two cases according to suitable conditions for the parameters $m$ and $h$ are satisfied, precisely:
\begin{enumerate}
    \item [$(i)$] $m=1$ and $h \in \F_{q^{2t}}$ with $\mathrm{N}_{q^{2t}/q^t}(h)=-1$
    \item [$(ii)$] $m$ neither a $(q-1)$-th power nor a $(q+1)$-th power of an element belonging to $\ker \mathrm{Tr}_{q^{2t}/q^t}$ and $q \equiv 1 \pmod 4$ if $t$ is odd.
    \end{enumerate} 
Actually, for $q \leq 5$ odd and $t\in\{3,4,5,6\}$, \texttt{Magma} computation suggests a wider choice of parameters for which the linear set $L_{s,m,h}^{5,t}$ is scattered (see \cite{GGLT}).

\subsubsection*{The linear sets of the type $L_{s,m,h}^{5,t}$, $m=1$ }
In order to investigate the geometrical properties of the projection vertex of a linear set of type $L_{s,h}^{5,t}:=L_{s,1,h}^{5,t}$ in $\PG(1,q^{2t})$, we first observe that, by Theorem \ref{thm:embed}, \cite[Proposition 3.1]{longobardi_zanella} and \cite[Proposition 4.17]{neri_santonastaso_zullo}, for any $t \geq 5$  $\psi_{t,h,s}(x):=\psi_{t,s,1,h}(x)$ is $\GaL(2t,q^{2t})$-equivalent to its adjoint. For $t=3$, using the same argument as done in \cite[Theorem 4.7]{longobardi_marino_trombetti_zhou}, the simplicity of $L_{s,h}^{5,t}$ follows. Therefore,  we get the following.
\begin{theorem}\label{quadsimple}
   The linear set $L_{s,h}^{5,t}$  of $\PG(1,q^{2t})$ is simple, for $t =3$ or $t \geq 5$.
\end{theorem}

Now, as we did for the previous examples, we prove the following result.
\begin{theorem}
    Let $L$ be an $(1,2(t-1))$-evasive linear set of rank $2t$ in a line $\Lambda$ of $\PG(2t-1, q^{2t})$. Let $\Sigma$ be a canonical subgeometry of $\PG(2t-1, q^{2t})$, $q$ odd and $t \geq 3$. Then $L$ is equivalent to a scattered linear set of the type $L_{s,h}^{5,t}$ if and only if for each $(2t-3)$ dimensional subspace $\Gamma$ of $\PG(2t-1, q^{2t})$ such that $L = \mathrm{p}_{\Gamma ,\Lambda}(\Sigma)$, the following holds:
\begin{enumerate}
\item [\textit{1.}] there exist a generator $\sigma$ of the subgroup of $\PGaL(2t, q^{2t})$ fixing $\Sigma$ pointwise and a  point $P=\langle v \rangle_{\F_{q^{2t}}} \in \PG(2t-1,q^{2t})$ such that
\begin{equation*}
    \Gamma = \langle P^{\sigma^i},Q, R, S : i \notin \{0,1,t-1,t+1,2t-1\} \rangle
\end{equation*} 
with $Q \in \langle P^{\sigma},P^{\sigma^{2t-1}} \rangle$, $R \in \langle P^{\sigma^{t-1}},P^{\sigma^{2t-1}} \rangle$ and $S \in \langle P^{\sigma^{t+1}}, P^{\sigma^{2t-1}} \rangle$,
\item [\textit{2.}]  the point $X= \langle v^{\sigma}-v^{\sigma^{t-1}} \rangle_{\F_{q^{2t}}}$ belongs to $\Gamma$, 
\item [\textit{3.}] there exists $h \in \mathcal{H}=\{z \in \F_{q^{2t}} \, \colon \, \mathrm{N}_{q^{2t}/q^t}(z)=-1\}$ such that
\begin{enumerate}

   \item  the points $P^{\sigma},P^{\sigma^{2t-1}},D,Q$ where $D=\langle v^\sigma + v^{\sigma^{2t-1}} \rangle_{\F_{q^{2t}}}$ are collinear  and the cross-ratio  $(P^{\sigma};P^{\sigma^{2t-1}};D;Q)=-h^{1-q^{s(2t-1)}}$,
   \item the points $P^{\sigma},P^{\sigma^{t+1}},R^{\sigma^2},Y$, where $Y=\langle Q,S\rangle \cap \langle P^{\sigma}, P^{\sigma^{t+1}}\rangle$, are collinear and the cross-ratio $(P^{\sigma};P^{\sigma^{t+1}};R^{\sigma^2};Y)=h^{1+q^{2s}}$.
            \end{enumerate}
\end{enumerate}
\end{theorem}
\begin{proof}

 Let us suppose that $L$ is equivalent to a maximum scattered linear set of type  $L^{5,t}_{s,h}$. By Theorem \ref{thm:move-gamma}, there exists $\beta$ such that $\boldsymbol{\Sigma}^\beta=\Sigma$ and either $\Gamma_{\psi_{t,h,s}}^\beta=\Gamma$  or $\Gamma_{\hat{\psi}_{t,h,s}}^\beta=\Gamma$. If $t=3$ or $t \geq 5$, since $L^{5,t}_{s,h}$ is simple, it is enough to prove the statement supposing that $\Gamma_{\psi_{t,h,s}}^\beta=\Gamma$.\\
Let $\boldsymbol{\rho}=\boldsymbol{\sigma}^s$ and $\boldsymbol{P}=\langle \textbf{v} \rangle_{\F_{q^{2t}}}$ be an imaginary point of $\PG(2t-1,q^{2t})$ (wrt $\boldsymbol{\Sigma}$), $\boldsymbol{Q} \in \langle \boldsymbol{P}^{\boldsymbol{\rho}}, \boldsymbol{P}^{\boldsymbol{\rho}^{2t-1}} \rangle $, $\boldsymbol{R} \in \langle \boldsymbol{P}^{\boldsymbol{\rho}^{t-1}}, \boldsymbol{P}^{\boldsymbol{\rho}^{2t-1}} \rangle$ and $\boldsymbol{S} \in \langle \boldsymbol{P}^{\boldsymbol{\rho}^{t+1}}, \boldsymbol{P}^{\boldsymbol{\rho}^{2t-1}} \rangle$ and
\begin{equation*}
    \Gamma_{\psi_{t,h,s}}=\langle \boldsymbol{P}^{\boldsymbol{\rho}^{i}}, \boldsymbol{Q}, \boldsymbol{R}, \boldsymbol{S}: i \not \in  \{1,t-1,t+1,2t-1\} \rangle
\end{equation*} 
where
\begin{equation*}
\boldsymbol{Q}=\langle \textbf{v}^{\boldsymbol{\rho}^{2t-1}}-h^{1-q^{s(2t-1)}}\textbf{v}^{\boldsymbol{\rho}}\rangle_{\F_{q^{2t}}}, \boldsymbol{R}= \langle \textbf{v}^{\boldsymbol{\rho}^{2t-1}}-h^{1-q^{s(2t-1)}}\textbf{v}^{\boldsymbol{\rho}^{t-1}}\rangle_{\F_{q^{2t}}}, \textnormal{ and }
\end{equation*}
\begin{equation*}
\boldsymbol{S}=\langle h^{1-q^{s(t+1)}}\textbf{v}^{\boldsymbol{\rho}^{2t-1}}+h^{1-q^{s(2t-1)}}\textbf{v}^{\boldsymbol{\rho}^{t+1}}\rangle_{\F_{q^{2t}}}.
\end{equation*}

 It is straightforward to check that the lines $\langle \boldsymbol{Q}, \boldsymbol{R} \rangle$ and $\langle \boldsymbol{P}^{\boldsymbol{\rho}},\boldsymbol{P}^{\boldsymbol{\rho}^{t-1}}\rangle$ meet at the point $\boldsymbol{X}=\langle \textbf{v}^{\boldsymbol{\rho}^{t-1}} -\textbf{v}^{\boldsymbol{\rho}}\rangle_{\F_{q^{2t}}}$ and this belongs to $\Gamma_{\psi_{t,h,s}}$. Also, the lines $\langle \boldsymbol{Q}, \boldsymbol{S} \rangle$ and $\langle \boldsymbol{P}^{\boldsymbol{\rho}}, \boldsymbol{P}^{\boldsymbol{\rho}^{t+1}} \rangle$  meet at the point $\boldsymbol{Y}=\langle h^{1-q^{s(t+1)}}\textbf{v}^{\boldsymbol{\rho}}+\textbf{v}^{\boldsymbol{\rho}^{t+1}}\rangle_{\mathbb{F}_{q^{2t}}}$. Let  $\boldsymbol{D}= \langle \textbf{v}^{\boldsymbol{\rho}} + \textbf{v}^{\boldsymbol{\rho}^{2t-1}} \rangle_{\F_{q^{2t}}}$, then the points $\boldsymbol{P}^{\boldsymbol{\rho}},\boldsymbol{P}^{\boldsymbol{\rho}^{2t-1}},\boldsymbol{D},\boldsymbol{Q}$ and $\boldsymbol{P}^{\boldsymbol{\rho}},\boldsymbol{P}^{\boldsymbol{\rho}^{t+1}},\boldsymbol{R}^{\boldsymbol{\rho}^2},\boldsymbol{Y}$ are collinear. Moreover, 
 \begin{equation*}
(\boldsymbol{P}^{\boldsymbol{\rho}};\boldsymbol{P}^{\boldsymbol{\rho}^{2t-1}};\boldsymbol{D};\boldsymbol{Q})  = \frac{\begin{vmatrix}
        1 & 1 \\
        0 & 1
   \end{vmatrix}
   \begin{vmatrix} 
    0&  1\\
    1 & -h^{1-q^{s(2t-1)}}
   \end{vmatrix}}
   {\begin{vmatrix}
   1 & 1 \\
   0 & -h^{1-q^{s(2t-1)}}
   \end{vmatrix}
   \begin{vmatrix}
       0 & 1 \\
       1 & 1
   \end{vmatrix}}=-h^{1-q^{s(2t-1)}}
 \end{equation*}
and 
 \begin{equation*}
(\boldsymbol{P}^{\boldsymbol{\rho}};\boldsymbol{P}^{\boldsymbol{\rho}^{t+1}};\boldsymbol{R}^{\boldsymbol{\rho}^2};\boldsymbol{Y})=
\frac{ \begin{vmatrix}
1 & 1 \\
0 &  -h^{q^{2s}-q^s}\\
 \end{vmatrix}
 \begin{vmatrix}
0 & h^{1-q^{s(t+1)}} \\
1 & 1
 \end{vmatrix}}
 {\begin{vmatrix}
1 & h^{1-q^{s(t+1)}} \\
0 & 1
 \end{vmatrix}
 \begin{vmatrix}
0 & 1 \\
1 & -h^{q^{2s}-q^s}\\
 \end{vmatrix}}= h^{q^{2s}+1}.
 \end{equation*}
Therefore, since $\beta$ preserve the collinearity and if $(A_1;A_2;A_3;A_4) \in \mathcal{H}$ then $(A^\beta_1;A^\beta_2;A^\beta_3;A^\beta_4)$ belongs to $\mathcal{H}$, the Points \textit{1-3} hold true.\\
For $t=4$, we have  
\begin{equation*}
\Gamma_{\hat{\psi}_{t,h,s}}=\langle \boldsymbol{P}^{\boldsymbol{\rho}^i},\boldsymbol{Q}, \boldsymbol{R}, \boldsymbol{S} : i \not \in \{1,3,5,7\}\rangle
\end{equation*}
where
\begin{equation*}
\boldsymbol{Q}=\langle \textbf{v}^{\boldsymbol{\rho}}-h^{q^s-1}\textbf{v}^{\boldsymbol{\rho}^7}\rangle_{\F_{q^{8}}},\quad \boldsymbol{R}= \langle \textbf{v}^{\boldsymbol{\rho}^{3}}+h^{q^{3s}-1}\textbf{v}^{\boldsymbol{\rho}^{7}}\rangle_{\F_{q^{8}}}, \textnormal{ and }
\end{equation*}
\begin{equation*}
\boldsymbol{S}=\langle \textbf{v}^{\boldsymbol{\rho}^{7}}-\textbf{v}^{\boldsymbol{\rho}^{5}}\rangle_{\F_{q^{8}}}.
\end{equation*}
If $\boldsymbol{\Sigma}^\beta=\Sigma$ and $\Gamma_{\hat{\psi}_{t,h,s}}^\beta=\Gamma$, let $\boldsymbol{D}=\langle \textbf{v}^{\rho} + \textbf{v}^{\boldsymbol{\rho}^{7}} \rangle_{\F_{q^{8}}}$, $\boldsymbol{T}=\langle  h^{q^{3s}-q^s}\textbf{v}^{\boldsymbol{\rho}} -\textbf{v}^{\boldsymbol{\rho}^3} \rangle_{\F_{q^{8}}}$ and $\boldsymbol{Z}=\langle \textbf{v}^{\boldsymbol{\rho}}-h^{1-q^{s}} \textbf{v}^{\boldsymbol{\rho}^5}\rangle_{\F_{q^{8}}}$. The result follows noting that 
\begin{equation*}
\Gamma_{\hat{\psi}_{t,h,s}}= \langle \boldsymbol{P}^{\boldsymbol{\rho}^i},\boldsymbol{Q}, \boldsymbol{Z}, \boldsymbol{T} : i \not \in \{1,3,5,7\}\rangle
\end{equation*}
and putting $\sigma=\beta^{-1}\boldsymbol{\rho}^{7}\beta$, $P=\boldsymbol{P}^\beta$, $Q=\boldsymbol{Q}^\beta$, $R=\boldsymbol{Z}^\beta$, $S=\boldsymbol{T}^\beta$, $X=\boldsymbol{S}^\beta$ and $D=\boldsymbol{D}^\beta$.\\
Now, let  suppose that Points $\textit{1}-\textit{3}$ hold true. By Theorem \ref{inverseline}, there exist a subgeometry $\Sigma$, a collineation $\sigma$ fixing $\Sigma$ pointwise and an imaginary point (wrt $\Sigma$) $P= \langle v \rangle_{\F_{q^{2t}}}$,  such that $L \cong L_f$ with $f(x)=x^{q^s}+\alpha x^{q^{s(t-1)}}+\beta x^{q^{s(t+1)}}+\gamma x^{q^{s(2t-1)}}\in \F_{q^n}[x]$, where $Q=\langle v^{\sigma^{2t-1}}-\gamma v^{\sigma}\rangle_{\F_{q^{2t}}}$, $R= \langle  \alpha v^{\sigma^{2t-1}}-\gamma v^{\sigma^{t-1}}\rangle_{\F_{q^{2t}}}$ and $S=\langle \beta v^{\sigma^{2t-1}}- \gamma v^{\sigma^{t+1}}\rangle_{\F_{q^{2t}}}$. Again, it is straightforward to see that  $X=\langle v^{\sigma^{t-1}} -\alpha v^{\sigma}\rangle_{\F_{q^{2t}}}$ and, hence by \textit{2}, $\alpha=1$.
By Point \textit{3(b)}, saying $D= \langle v^\sigma + v^{\sigma^{2t-1}} \rangle_{\F_{q^{\textcolor{blue}{2}t}}}$, we get $\gamma=h^{1-q^{s(2t-1)}}$ for some $h \in \F_{q^{2t}}$ with $\mathrm{N}_{q^{2t}/q^t}(h)=-1$. Finally, since $Y=\langle \beta v^{\sigma}-v^{\sigma^{t+1}}\rangle_{\mathbb{F}_{q^{2t}}}$,  \textit{3(c)} implies that $\beta \gamma^{q^{2s}}=h^{1+q^{2s}}$, so $\beta=\frac{h^{1+q^{2s}}}{h^{q^{2s}-q^s}}=h^{1+q^s}=-h^{1-q^{s(t+1)}}$. Therefore, $L \cong L_{s,h}^{5,t}$.
 \end{proof}

By Point \textit{$3(a-b)$} of Theorem above, if $h \in \F_{q^t}$, then the cross-ratio $(P^\sigma,P^{\sigma^{2t-1}},D,Q)$ and $(P^\sigma,P^{\sigma^{t+1}},R^{\sigma^2},Y)$ are equal to $-1$. Thus, setting $L^{5,t}_{s}:=L^{5,t}_{h,s}$ with $h \in \F_{q^t}$, we can state the following.

\begin{corollary}
    Let $L$ be an $(1,2(t-1))$-evasive linear set of rank $2t$ in a line $\Lambda$ of $\PG(2t-1, q^{2t})$. Let $\Sigma$ be a canonical subgeometry of $\PG(2t-1, q^{2t})$, $q$ odd. Then $L$ is equivalent to a scattered linear set of the type $L_{s}^{5,t}$ if and only if for each $(2t-3)$ dimensional subspace $\Gamma$ of $\PG(2t-1, q^{2t})$ such that $L = \mathrm{p}_{\Gamma ,\Lambda}(\Sigma)$, the following holds:
\begin{enumerate}
\item [\textit{1.}] there exists a generator $\sigma$ of the subgroup of $\PGaL(2t, q^{2t})$ fixing $\Sigma$ pointwise and a  point $P=\langle v \rangle_{\F_{q^{2t}}} \in \PG(2t-1,q^{2t})$ such that
\begin{equation*}
    \Gamma = \langle P^{\sigma^i},Q, R, S : i \notin \{0,1,t-1,t+1,2t-1\} \rangle
\end{equation*} 
with $Q \in \langle P^{\sigma},P^{\sigma^{2t-1}} \rangle$, $R \in \langle P^{\sigma^{t-1}},P^{\sigma^{2t-1}} \rangle$ and $S \in \langle P^{\sigma^{t+1}}, P^{\sigma^{2t-1}} \rangle$,
\item [\textit{2.}] the point $X= \langle v^\sigma -v^{\sigma^{t-1}}\rangle_{\F_{q^{2t}}}$ belongs to $\Gamma$.
\item [\textit{3.}]  \begin{enumerate}
   \item  [\textit{(a)}] the pair $\{P^{\sigma},P^{\sigma^{2t-1}}\}$ separates $\{D,Q\}$ harmonically, where $D=\langle v^\sigma + v^{\sigma^{2t-1}} \rangle_{\F_{q^{2t}}}$. 
   \item[\textit{(b)}] the pair $\{P^{\sigma},P^{\sigma^{t+1}}\}$ separates $\{R^{\sigma^2},Y\}$ harmonically, where $Y=\langle Q,S\rangle \cap \langle P^{\sigma} , P^{\sigma^{t+1}} \rangle$.
\end{enumerate}
\item [\textit{4.}] If $t$ is odd, then $q \equiv 1 \pmod 4$.
\end{enumerate}
\end{corollary}

\subsubsection*{The linear sets of the type $L_{m,s}^{5,t}$, for $m \in \F_{q^t} \setminus \{0,1\}$}  
 In this subsection we will deal with the family of scattered linear sets of $\mathrm{PG}(1,q^{2t})$, $t \geq 3$ and $q$ odd, recently studied in \cite{SmaZaZu} (cf. \eqref{quadrinomial}).
The elements in this class have the following fashion:
\begin{equation*}
L_{m,s}^{5,t}=\{ \langle(x, m(x^{q^s}-x^{q^{s(t+1)}})+ x^{q^{s(t-1)}}+x^{q^{s(2t-1)}}\rangle_{\mathbb{F}_{q^{2t}}} : x \in \mathbb{F}_{q^{2t}}^* \},
\end{equation*}
where $\gcd(s,2t)=1$.  Here,  $m \in \mathbb{F}_{q^t}^*$ is neither a $(q-1)$-th power nor a $(q+1)$-th power of any element of the vector space $\ker \mathrm{Tr}_{q^{2t}/q^t}$ and if $t$ is odd, $q \equiv 1 \pmod 4$, see \cite[Theorem 2.3]{SmaZaZu}.
Firstly, by \cite[Theorem 4.5]{SmaZaZu}  and Theorem \ref{thm:embed}, for this class of scattered linear sets with  parameter $m \in \F_{q^t}^*$ as above, we can state the following.
\begin{theorem}
    For $t>4$, the $\GaL$-class of the scattered linear set $L_{m,s}^{5,t}$ is two.
\end{theorem}
  
Similarly to what did before, we are now ready to characterise the projection vertex of a linear set equivalent an element of the type $L_{m,s}^{5,t}$. 
    \begin{theorem}
    Let $L$ be an $(1,2(t-1))$-evasive linear set of rank $2t$ in a line $\Lambda$ of $\PG(2t-1, q^{2t})$. Let $\Sigma$ be a canonical subgeometry of $\PG(2t-1, q^{2t})$. Then $L$ is equivalent to a maximum scattered linear set of the type $L_{m,s}^{5,t}$ if and only if for each $(2t-3)$ dimensional subspace $\Gamma$ of $\PG(2t-1, q^{2t})$ such that $L = \mathrm{p}_{\Gamma ,\Lambda}(\Sigma)$, the following holds:
\begin{enumerate}
\item [\textit{1.}] there exist a generator $\sigma$ of the subgroup of $\PGaL(2t, q^{2t})$ fixing $\Sigma$ pointwise and a  point $P=\langle v \rangle_{\F_{q^{2t}}}\in \PG(2t-1,q^{2t})$ such that
\begin{equation*}
    \Gamma = \langle P^{\sigma^i},Q, R, S : i \notin \{0,1,t-1,t+1,2t-1\} \rangle
\end{equation*} 
with $Q \in \langle P^{\sigma},P^{\sigma^{2t-1}} \rangle$, $R \in \langle P^{\sigma^{t-1}},P^{\sigma^{2t-1}} \rangle$ and $S \in \langle P^{\sigma^{t+1}}, P^{\sigma^{2t-1}} \rangle$,
     \item [\textit{2.}] the point $C=\langle v^{\sigma}+v^{\sigma^{t+1}} \rangle_{\mathbb{F}_{q^{2t}}} \in \Gamma$,
\item [\textit{3.}] there exists $\mu \in \mathcal{W}= \{ z \in \F_{q^{t}} \colon   \nexists\, w \in \ker \mathrm{Tr}_{q^{2t}/q^t} \textnormal{ s.t. } z=w^{q-1} \textnormal{ or } z=w^{q+1} \}$  and
 \begin{enumerate}

     \item the points $P^\sigma, P^{\sigma^{2t-1}}$, $X$, $Q$, where $X= \langle v^{\sigma}-v^{\sigma^{2t-1}}\rangle_{\F_{q^{2t}}}$, are collinear  and the cross ratio $(P^{\sigma};P^{\sigma^{2t-1}};X;Q)=\mu$,
\item the points $P^{\sigma^{t-1}}, P^{\sigma^{2t-1}}$, $Y$, $R$, where $Y= \langle v^{\sigma^{t-1}}+v^{\sigma^{2t-1}}\rangle_{\F_{q^{2t}}}$, are collinear and the pair $\{ P^{\sigma^{t-1}}, P^{\sigma^{2t-1}} \}$ separates the pair $\{ Y, R \}$ harmonically.
 \end{enumerate}
\end{enumerate}
\end{theorem}
\begin{proof}
 Let us suppose that $L$ is equivalent to $L_{m,s}^{5,t}$.
  By Theorem \ref{thm:move-gamma}, there exists $\beta$ such that $\boldsymbol{\Sigma}^\beta=\Sigma$, saying $\psi_{t,m,s}:=\psi_{t,m,h,s}$ with $h \in \F_{q^t}$, and either $\Gamma_{\psi_{t,m,s}}^\beta=\Gamma$  or $\Gamma_{\hat{\psi}_{t,m,s}}^\beta=\Gamma$. \\
If  $\beta \in \PGaL(2t,q^{2t})$ such that $\Gamma_{\psi_{t,m,s}}^\beta=\Gamma$, then let $\boldsymbol{\rho}=\boldsymbol{\sigma}^s$ with $\gcd(s,2t)=1$, $\boldsymbol{P}=\langle \textbf{v} \rangle_{\F_{q^{2t}}}$ be an imaginary point of $\PG(2t-1,q^{2t})$ (wrt $\boldsymbol{\Sigma}$) and
\begin{equation*}
    \Gamma_{\psi_{t,m,s}}=\langle \boldsymbol{P}^{\boldsymbol{\rho}^{i}}, \boldsymbol{Q}, \boldsymbol{R}, \boldsymbol{S}: i \not \in  \{1,t-1,t+1,2t-1\} \rangle
\end{equation*} 
where
\begin{equation*}
\boldsymbol{Q}=\langle m\textbf{v}^{\boldsymbol{\rho}^{2t-1}}-\textbf{v}^{\boldsymbol{\rho}}\rangle_{\F_{q^{2t}}},\quad \boldsymbol{R}= \langle \textbf{v}^{\boldsymbol{\rho}^{2t-1}}-\textbf{v}^{\boldsymbol{\rho}^{t-1}}\rangle_{\F_{q^{2t}}}, \textnormal{ and }
\end{equation*}
\begin{equation*}
\boldsymbol{S}=\langle m\textbf{v}^{\boldsymbol{\rho}^{2t-1}}+\textbf{v}^{\boldsymbol{\rho}^{t+1}}\rangle_{\F_{q^{2t}}}.
\end{equation*}
Let us consider  the point $\boldsymbol{C}=\langle \textbf{v}^{\boldsymbol{\rho}}+\textbf{v}^{\boldsymbol{\rho}^{t+1}}\rangle_{\F_{q^{2t}}}$; it belongs to $\Gamma_{\psi_{t,m,s}}$. Moreover, let $\boldsymbol{X}=\langle \textbf{v}^{\boldsymbol{\rho}} -\textbf{v}^{\boldsymbol{\rho}^{2t-1}} \rangle_{\F_{q^{2t}}}$.   Clearly, $\boldsymbol{X},\boldsymbol{Q} \in \langle \boldsymbol{P}^{\boldsymbol{\rho}};\boldsymbol{P}^{\boldsymbol{\rho}^{2t-1}} \rangle$ and $(\boldsymbol{P}^{\boldsymbol{\rho}};\boldsymbol{P}^{\boldsymbol{\rho}^{2t-1}};\boldsymbol{X};\boldsymbol{Q}) = 1/m$. Finally,  $(\boldsymbol{P}^{\boldsymbol{\rho}^{t-1}};\boldsymbol{P}^{\boldsymbol{\rho}^{2t-1}}; \boldsymbol{Y};\boldsymbol{R} )=-1$ where $\boldsymbol{Y}=\langle \textbf{v}^{\boldsymbol{\rho}^{t-1}} + \textbf{v}^{\boldsymbol{\rho}^{2t-1}} \rangle_{\F_{q^{2t}}}$. By setting  $\sigma=\beta^{-1}\boldsymbol{\rho}\beta$, $P=\boldsymbol{P}^\beta$, $Q=\boldsymbol{Q}^\beta$, $R=\boldsymbol{R}^\beta$, $S=\boldsymbol{S}^\beta$, $C=\boldsymbol{C}^\beta$ $X=\boldsymbol{X}^\beta$ and $Y=\boldsymbol{Y}^\beta$ and noting that $\beta$ preserves the collinearity, being harmonically separated and if
$(A_1;A_2;A_3;A_4) \in \mathcal{W}$ then $(A^\beta_1;A^\beta_2;A^\beta_3;A^\beta_4)\in \mathcal{W}$, the Point\textit{(1--3)} follow.\\
If $\Gamma_{\hat{\psi}_{t,m,s}}^\beta=\Gamma$, by Remark \ref{aggiunto},
\begin{equation*}
    \Gamma_{\hat{\psi}_{t,m,s}}=\langle \boldsymbol{P}^{\boldsymbol{\rho}^i},\boldsymbol{Q},\boldsymbol{R}, \boldsymbol{S} : i \notin \{0,1,t-1,t+1,2t-1\} \rangle 
\end{equation*}  
where
\begin{equation*}
\boldsymbol{Q}=\langle m^{\boldsymbol{\rho}^{t-1}}\textbf{v}^{\boldsymbol{\rho}}+\textbf{v}^{\boldsymbol{\rho}^{t-1}}\rangle_{\F_{q^{2t}}},\quad \boldsymbol{R}= \langle \textbf{v}^{\boldsymbol{\rho}}-\textbf{v}^{\boldsymbol{\rho}^{t+1}}\rangle_{\F_{q^{2t}}}, \textnormal{ and }
\end{equation*}
\begin{equation*}
\boldsymbol{S}=\langle m^{\boldsymbol{\rho}^{2t-1}}\textbf{v}^{\boldsymbol{\rho}}-\textbf{v}^{\boldsymbol{\rho}^{2t-1}}\rangle_{\F_{q^{2t}}}
\end{equation*}
Let us consider  the point $\boldsymbol{C}=\langle \textbf{v}^{\boldsymbol{\rho}^{2t-1}}+\textbf{v}^{\boldsymbol{\rho}^{t-1}}\rangle_{\F_{q^{2t}}}$. Since $m \in \F_{q^t}$, $ \boldsymbol{C} \in\Gamma_{\hat{\psi}_{t,m,s}}$. Moreover, let consider the point $\boldsymbol{X}=\langle \textbf{v}^{\boldsymbol{\rho}} -\textbf{v}^{\boldsymbol{\rho}^{2t-1}} \rangle_{\F_{q^{2t}}}$. Similarly to the previous case, $\boldsymbol{X},\boldsymbol{S} \in \langle \boldsymbol{P}^{\boldsymbol{\rho}};\boldsymbol{P}^{\boldsymbol{\rho}^{2t-1}} \rangle$ and $(\boldsymbol{P}^{\boldsymbol{\rho}};\boldsymbol{P}^{\boldsymbol{\rho}^{2t-1}};\boldsymbol{X};\boldsymbol{S}) = m^{\boldsymbol{\rho}^{2t-1}}$. If $\boldsymbol{Y}= \langle \textbf{v}^{\boldsymbol{\rho}^{t+1}} + \textbf{v}^{\boldsymbol{\rho}} \rangle_{\F_{q^{2t}}}$,  the cross-ratio $(\boldsymbol{P}^{\boldsymbol{\rho}^{t-1}};\boldsymbol{P}^{\boldsymbol{\rho}^{2t-1}}; \boldsymbol{Y};\boldsymbol{R} )=-1$.  As we have done before, the result follows by setting  $\sigma=\beta^{-1}\boldsymbol{\rho}^{2t-1}\beta$, $P=\boldsymbol{P}^\beta$, $Q=\boldsymbol{S}^\beta$, $R=\boldsymbol{R}^\beta$, $S=\boldsymbol{Q}^\beta$, $C=\boldsymbol{C}^\beta$, $X=\boldsymbol{X}^\beta$ and $Y=\boldsymbol{Y}^\beta$. \\
Now, let us suppose that Point \textit{(1--3)} hold true. Then, by Theorem \ref{inverseline}, $L \cong L_f$ with $f(x)=\alpha x^{q^s}+\beta x^{q^{s(t-1)}}+\gamma x^{q^{s(t+1)}}+x^{q^{s(2t-1)}} \in \F_{q^{2t}}[x]$ for some $\alpha,\beta,\gamma \in \F_{q^{2t}}$ and some integer $1 \leq s \leq 2t-1$ such that $\gcd(s,2t)=1$. Furthermore there exist $P=\langle v \rangle _{\F_{q^{2t}}}$, an imaginary point with respect to the subgeometry $\Sigma$ fixed pointwise by a collineation $\sigma \in \PGaL(2t,q^{2t})$. Thus, let consider the points $Q=\langle \alpha v^{\sigma^{2t-1}}- v^{\sigma}\rangle_{\F_{q^{2t}}}$, $R= \langle \beta v^{\sigma^{2t-1}}- v^{\sigma^{t-1}}\rangle_{\F_{q^{2t}}}$ and $S= \langle \gamma v^{\sigma^{2t-1}}- v^{\sigma^{t+1}}\rangle_{\F_{q^{2t}}}$. Since $C=\langle v^{\sigma}+v^{\sigma^{t+1}} \rangle_{\F_{q^{2t}}}\in \Gamma$, then $\alpha=-\gamma$. By Point \textit{3(a)}, the cross-ratio $(P^{\sigma};P^{\sigma^{2t-1}};X;Q)=1/{\alpha} \in \mathcal{W}$. By Point \textit{3(b)}, we get $\beta=1$. Since, if $1/\alpha \in \mathcal{W}$, then $\alpha \in \mathcal{W}$, we get $f(x)=\psi_{t,\alpha,s}(x)$ and hence $L \cong L_{\alpha,s}^{5,t}.$
    \end{proof}

\section*{Acknowledgments}

The authors express their gratitude to Prof. Giuseppe Marino for helpful
discussions and technical advices. Moreover, they wish to thank
the Italian National Group for Algebraic and Geometric Structures
and their Applications (GNSAGA– INdAM) for supporting this research.

\vspace{1cm}
\noindent Dipartimento di Matematica e Applicazioni “Renato Caccioppoli”\\
Università degli Studi di Napoli Federico II,\\
Via Cintia, Monte S. Angelo I-80126 Napoli, Italy. \\
email addresses: \\
\texttt{\{giovannigiuseppe.grimaldi, somi.gupta, \\giovanni.longobardi, rtrombet\}@unina.it}

\end{document}